\documentclass{article}
\usepackage{latexsym,amsfonts,amsthm,amsmath,mathrsfs}
\usepackage{color}
\usepackage{enumerate}
\usepackage[margin=1.25in]{geometry}

\usepackage{amscd, epsfig}
\usepackage{fancyhdr, graphicx}
\usepackage{dsfont}
\usepackage{subfigure}
\usepackage{float}
\newtheorem{thm}{Theorem}[section]
\newtheorem{lemma}[thm]{Lemma}
\newtheorem{prop}[thm]{Proposition}
\newtheorem{corollary}[thm]{Corollary} %[section]
\theoremstyle{definition}
\newtheorem{defin}[thm]{Definition}
\newtheorem{rem}[thm]{Remark}%[section]
\numberwithin{equation}{section} %NEW

\newcommand{\R}{\mathbb{R}}
\newcommand{\C}{\mathbb{C}}

\newcommand{\N}{\mathbb{N}}
\newcommand{\la}{\lambda}
\newcommand{\vp}{\varphi}

\newcommand{\beq}{\begin{equation}}
\newcommand{\eeq}{\end{equation}}

\newcommand{\La}{\mathcal{L}}

\newcommand{\eps}{\varepsilon}

\newcommand{\rp}{\text{Re}}

\newcommand{\weak}{\rightharpoonup}
\newcommand{\E}{\mathscr{E}}
\newcommand{\e}{\varepsilon}

\newcommand{\T}{\mathbb{T}}

\title{Slow motion for the 1D Swift--Hohenberg equation}
\author{Gurgen Hayrapetyan\\
	Ohio University \\
	Athens, OH, USA \\
	email: hayrapet@ohio.edu
	\and Matteo Rinaldi\\
	Carnegie Mellon University \\
	Pittsburgh, PA, USA \\
	email: matteor@andrew.cmu.edu}

\begin{document}

\maketitle
\begin{abstract}
The goal of this paper is to study the behavior of certain solutions to the Swift--Hohenberg equation on a one--dimensional torus $\T$. Combining results from $\Gamma$--convergence and ODE theory, it is shown that solutions corresponding to initial data that is $L^1$--close to a jump function $v$, remain close to $v$ for large time. This can be achieved by regarding the equation as the $L^2$--gradient flow of a second order energy functional, and obtaining asymptotic lower bounds on this energy in terms of the number of jumps of $v$.

% This can be achieved by regarding the equation as the $L^2$--gradient flow of a fourth order energy functional and studying the asymptotic behavior of solutions of the corresponding Euler--Lagrange equation.
% 
%  {\color{blue} is it clear what we mean by tail, or should we remove this sentence? The linearization of such equation provides almost sharp estimates on the tail of the associated energy.}
\end{abstract}

\iffalse
\fi

\section{Introduction, Motivation and Main Results}

The fourth order partial differential equation 
\beq
\label{SHOrig}
u_t = ru - (\bar{q}^2 + \partial_{x}^2)^2 u + f(u)
\eeq
is a generalization of the Swift--Hohenberg equation introduced in 1977 by Swift and Hohenberg \cite{SwiftHohenberg} as a model for the study of pattern formation, in connection with the Rayleigh--B\'ernard convection, e.g. see \cite{CrossHohenberg},\cite{Kuwamura}. Among many different applications, the most famous ones in the literature are those in connection to pattern formation in vibrated granular materials \cite{UmbaMelo}, buckling of long elastic structures \cite{HuntPeletier}, Taylor--Couette flow \cite{SwiftHohenberg2}, \cite{PomeauManneville}, and in the study of lasers \cite{LegaMoloney}. Moreover, in recent years great attention has been paid to models of phase transitions in the study of pattern--formation in bilayer membranes, see e.g. \cite{ChermisiDMFL} where the Swift--Hohenberg equation turns out to be the gradient flow of Ginzburg--Landau type energies, with respect to the right inner product structure.

Consider \eqref{SHOrig} on a periodic domain with a characteristic size $L = 1/\eps$, where $0 < \e \ll 1$.  Letting $W$ be the primitive of $s \mapsto 2(f(s) +(r-\bar{q}^4) s)$, $q:= 2\bar{q}^2$, and rescaling time and space by $\eps$ in \eqref{SHOrig} one arrives at the rescaled form
\beq
\label{PDE}
\begin{cases}
u_t = -W'(u) - 2 \eps^2 q u_{xx} - 2 \eps^4 u_{xxxx} & x \in \T, t > 0,\\
%u_t = - W'(u) - 2 \eps^2 q u'' - 2 \eps^4 u^{(iv)}, \quad a<x<b
u(x,0) = u_{0,\e} (x) & x \in \T,\\
\end{cases}
\eeq
where $\T$ is a one--dimensional torus.  We assume that $W : \R \to [0,+\infty)$ is a double--well potential with phases supported at $-1$ and $1$, and we study the long--time behavior of solutions when $q > 0$ is sufficiently small. In particular, due to the presence of the small parameter $\eps$ in  \eqref{PDE} the solutions are expected to develop interfacial structure driven by the minima of the potential $W$.
Equation \eqref{PDE} may be viewed as a gradient flow associated to a second order energy functional, and our main result consists of an asymptotic lower bound on the corresponding energy functional and the consequent bounds on the speed of evolution of the developed interfaces.  
%In subsections below we outline 
%known results for the 
%Swift-Hohenberg equation, as well as
Below we outline interfacial dynamics results for the lower order Allen--Cahn equation and its generalizations that provide much of the motivation for our analysis.

%\subsection{Swift-Hohenberg Equation}

\subsection{Allen--Cahn Equation and Generalizations to Higher Order}
Equations displaying interfacial dynamics have been studied extensively in the last two decades. 
The prototypical example is the Allen--Cahn equation %(see \cite{AllenCahn})
\beq \label{allencahn}
u_t = \e^2 u_{xx} - W'(u), \quad x \in I, \,  t > 0,
\eeq
(as well as its higher dimensional analog)
seen as the $L^2$--gradient flow of the energy
\beq \label{BKenergy}
G_\e(u; I):= \int_I \left( \frac{1}{\e}W(u) + \frac{\e}{2}|u_x|^2 \right) dx , \quad u \in H^1(I),
\eeq
where $I \subset \R$ is an interval. The special gradient--flow structure of \eqref{allencahn} has allowed its analysis by a wide variety of methods and techniques. 

In particular, it has been shown for the Allen--Cahn equation (see \cite{chen1992generation} and the references therein) that if $\eps \ll 1$ the evolution from a sufficiently regular initial data occurs in four main stages.  
In the first stage, the diffusion term $\e^2 u_{xx}$ can be ignored and the leading order dynamics are driven by the $\e$ independent ordinary differential equation $u_t = - W'(u)$.  This is the time-scale in which interfaces develop, i.e., regions in the space domain that separate almost constant solutions corresponding to the stable equilibria of the ordinary differential equation.  This stage, referred to as the generation of interface, has been analyzed for the Allen--Cahn equation first in \cite{FifeHsiao}, and subsequently in  \cite{chen1992generation}, \cite{chen1991dynamics}, \cite{deMottoniSchatzman}, \cite{Soner1997}, and other papers.

As the regions separating unequal equilibria decrease in length, the spacial gradient necessarily increases, and after $O(|\ln \e|)$ time the dynamics are driven by a balance between the two terms on the right--hand side of \eqref{allencahn}.  In particular, as shown in \cite{chen1992generation}, after $O(\eps^{-1})$ time the solution is exponentially close to the standing--wave profile
\beq
\Phi(x; p_1, \dots, p_n) :=  \pm \prod_n \phi\left( \frac{x-p_i}{\e}\right), 
\eeq
parametrized by the positions $p_1, \dots p_n$, where $\phi$ satisfies
\beq
\phi'' = W'(\phi), \quad \lim_{z\rightarrow \pm \infty} \phi(x) = \pm 1, \quad \phi(0) = 0.
\eeq
The zeros $p_1(t), \dots, p_n(t)$ of $\Phi$ can be viewed as specifying the location of the interfaces.  In particular, the residual $\eps^2 \Phi_{xx} - W'(\Phi)$ is exponentially small and the corresponding third stage of the evolution proceeds on an exponentially slow time scale until two zeros of the solution of \eqref{allencahn} $u^\e$ collide and disappear as part of the fourth stage of the evolution.  

The third stage, usually referred to as \emph{Slow Motion} has been studied extensively.  The most precise interface evolution results for the Allen--Cahn equation are given in  \cite{CarrPego2}, \cite{CarrPego1}, \cite{Fusco}, \cite{FuscoHale}. Specifically, the zeros of the solution $u^\e$ are approximated by $\{p_i\}$, which at leading order move according to the evolution law
\beq
\label{CPEvol}
p'_i = \eps S \left( \exp\left( -\mu \frac{p_{i+1} - p_{i}}{\e}\right) - \exp\left( -\mu \frac{p_{i} - p_{i-1}}{\e}\right) \right),
\eeq
where $\mu = \sqrt{W''(\pm 1)}$, $S > 0$ is a constant depending only on $W$.
The proof of this reduction involves invariant manifold theory and geometric analysis. 

In \cite{BronsardKohn} Bronsard and Kohn adopted a variational viewpoint to study the Allen--Cahn equation.  While their method does not recover the evolution equation above, it does provide relatively simple energy arguments to obtain a bound on the speed of this evolution.
%
%They gave a partial variational validation of results previously found, see  \cite{CarrPego2}, \cite{CarrPego1}, \cite{FuscoHale} and \cite{Fusco}, which is closer in spirit to the study of \eqref{BKenergy} via $\Gamma$--convergence, see  \cite{ModicaMortola}, \cite{Modica}, \cite{Sternberg}.
In particular, Bronsard and Kohn first prove that for any $k > 0$ there exists a constant $c_k > 0$ such that, if $v \in H^1(I)$ is sufficiently close in $L^1$ norm to a step function taking values $\pm 1$ and having exactly $N$ jumps, and its energy satisfies 
\beq
\label{BKEH}
G_\e(v;I) \le N c_W + \eps^k,
\eeq
where $c_W = \int_{-1}^1 \sqrt{2 W(s)} ds$, then
\beq
\label{BKlower}
G_{\e}(v;I) \geq Nc_W - c_k \e^{k}.
\eeq
Using this energy estimate they prove that the solution $u^\e$ of \eqref{allencahn} with Dirichlet or Neumann boundary data, under the same conditions on the initial data $u_{0,\e}(x)$, satisfies
\beq \label{SlowDefinition}
\lim_{\e \to 0} \sup_{0 \leq t \leq \e^{-k}m} \int_{I} |u^\e(x,t) - v(x)| dx = 0,
\eeq 
for any $m > 0$.
The limit in \eqref{SlowDefinition} may be viewed as providing an upper bound on the speed of the evolution of the transition layers of $u^\e$. Improvements of \eqref{BKlower} have been obtained in \cite{BellettiniNayamNovaga} and \cite{Grant}.  In particular, it has recently been established in \cite{BellettiniNayamNovaga} that for a sequence $\{v_\e\} \subset H^1(\T)$ converging to a step function taking values $\pm 1$ and having exactly $N$ jumps, the Allen--Cahn functional admits the following multiple order asymptotic expansion
\[ \label{BNN}
\begin{aligned}
G(v_\eps; \T) &= N c_W - 2 \alpha_+ \kappa_+^2 \sum_{k=1}^{N} \exp \left(-\alpha_+ \frac{d_k^\eps}{\eps} \right) - 2 \alpha_- \kappa_-^2 \sum_{k=1}^{N} \exp \left(-\alpha_- \frac{d_k^\eps}{\eps}\right) \\
&+ \kappa_+^3 \beta_+ \sum_{k=1}^{N} \exp \left(\frac{-3\alpha_+}{2} \frac{d_k^\eps}{\eps}\right) + \kappa_-^3 \beta_- \sum_{k=1}^{N} \exp \left(\frac{-3\alpha_-}{2} \frac{d_k^\eps}{\eps}\right) \\
&+ o\left( \sum_{k=1}^{N} \exp \left(\frac{-3\alpha_+}{2} \frac{d_k^\eps}{\eps}\right) \right) + o\left( \sum_{k=1}^{N} \exp \left(\frac{-3\alpha_-}{2} \frac{d_k^\eps}{\eps}\right)  \right)
\end{aligned}
\]
where $\alpha_{\pm}, \kappa_{\pm}, \beta_{\pm}$ are constants dependent on the potential $W$ and $d_k^\eps$ is the distance between consecutive transition layers of $v_\eps$. 
 The gradient flow associated with the second order term in the above energy expansion gives, up to a multiplicative constant, the evolution equation \eqref{CPEvol}, providing a crucial link between the variational and geometric approaches.  Further insight into this connection can be seen as part of a general framework of $\Gamma$--convergence of gradient flows developed in \cite{SandierSerfaty}. \\

In regards to extensions to higher--order functionals, the problem has been studied in \cite{KaliesVanderVorstWanner} in connection with a family of higher order functionals of the form 
\beq
\mathcal{H}(u) := \frac{1}{\eps} \int_I \left (\sum_{k-1}^n \frac{\gamma_k \eps^{2k}}{2} | u^{(k)} |^2 + W(u) \right) dx,
\eeq
where $u^{(k)}$ stands for the $k$--th spatial derivative of $u$. Due to difficulties associated with higher order nature of the functional, in particular, the lack of exact solutions of the corresponding Euler--Lagrange equation, sharp bounds analogous to \eqref{BNN} have not been established.
An important condition on $\mathcal{H}$ in \cite{KaliesVanderVorstWanner} is
\begin{itemize}
\item Hypothesis 1: There exists constants $d_0, \eta > 0$ such that for every interval $I \subset \R$ with length $|I| \ge d_0$ and all $u \in H^n(I)$ 
\beq
\int_I \left( \sum_{k-1}^n \gamma_k | u^{(k)} |^2 \right) dx \ge \eta \int_I \left( | u^{(n)} |^2 + |u'|^2 \right) dx.
\eeq
\end{itemize}
Under this hypothesis the authors prove that for any $u \in H^n(I)$ sufficiently close to a step function  taking values $\pm 1$ and having exactly $N$ jumps,
\beq
\label{kalB}
\mathcal{H}_\e(u) \ge N m_1 - C\exp \left(-\frac{d \lambda}{3\e}\right),
\eeq
where $\lambda$ is a constant satisfying $\lambda < |{\rm Re}(\mu)|$, for all eigenvalues $\mu$ of the linearization of 
\beq
\sum_{k=1}^n (-1)^k \gamma_k u^{(2k)} + W'(u) = 0
\eeq
at $(\pm 1, 0, \dots, 0)$. \\
%In turn, a more precise asymptotic analysis of the energy functional leads to a more precise analysis of the motion of solution of the corresponding gradient flow. 

The initial value problem \eqref{PDE} can be seen as the $L^2$--gradient flow of the second order energy functional
\beq
\label{ourEnergy}
E_{\eps}(u; \T) := \int_\T \left( \frac{1}{\eps} W(u) - \eps q |u_x|^2 + \eps^3 |u_{xx}|^2 \right)  dx, \quad u \in H^2(\T)
\eeq
and our main goals are the extension and the improvement of the bound \eqref{kalB} for this energy and, in turn, this will allow us to prove the slow motion of solutions of \eqref{PDE}. \\

\iffalse
Our main result is the extension of this bound to the functional 
\beq  \label{ourEnergy}
E_{\eps}(u; \T) := \int_\T \left( \frac{1}{\eps} W(u) - \eps q |u_x|^2 + \eps^3 |u_{xx}|^2 \right)  dx, \quad u \in H^2(\T),
\eeq
the $L^2$--gradient flow of the energy \eqref{ourEnergy}.  
\fi

We note that  the functional \eqref{ourEnergy} does not satisfy Hypothesis 1  due to the negative term in the energy.  We use  recently established interpolation inequality (see \cite{ChermisiDMFL} and \cite{CicaleseSpadaroZeppieri}) to overcome this difficulty if $q > 0$ is sufficiently small.  Moreover, in the proof of an energy estimate analogous to \eqref{kalB}, see Theorem \ref{thm1}, we do not assume any closeness condition on the $H^2$ functions we work with, we instead make an assumptions on the zeros of such functions. %that $u$ is close to a step function, only that it has $N$ zeros. 

Furthermore, inspired by \cite{BellettiniNayamNovaga}, our analysis relies on the use of a particular test function, and on the study of the solutions of the Euler--Lagrange equation associated to \eqref{ourEnergy} via hyperbolic fixed point theory, in particular through the work of Sell \cite{Sell}. Thanks to this approach we are able to  improve the exponent in \eqref{kalB} and, consequently, obtain sharper bound on the speed of evolution for solutions of \eqref{PDE}.

We recall that the $\Gamma$--convergence of the energy functional $E_\e$ has been proved in \cite{FonsecaMantegazza} for the case $q= 0$, and in \cite{ChermisiDMFL} and \cite{CicaleseSpadaroZeppieri} when $q > 0$ is small. The asymptotic behavior of $E_\e$ plays a crucial role in our analysis: we will use results from $\Gamma$--convergence, together with a careful analysis of the minimizers of the associated Euler--Lagrange equation, to study the speed of motion of solutions of \eqref{PDE}.

To conclude, we remark that the situation in the higher dimensional setting is quite different: solutions of the higher dimensional version of \eqref{allencahn} and other classical gradient flow--type equations have been studied by many different authors, see, e.g.,  \cite{AlikakosBronsardFusco}, \cite{AlikakosFusco}, \cite{BronsardKohn2}, \cite{EiYanagida}, \cite{Kowalczyk}, \cite{OttoRez}. Due to the lack of results like \eqref{BKlower}, all of them use significantly different approaches to the one introduced in \cite{BronsardKohn}. A more recent work, see \cite{MurrayRinaldi}, closes the gap by making use of a $\Gamma$--convergence result proved in \cite{LeoniMurray} and doesn't assume any specific structure of the initial data. \\

\subsection{Statement of Main Results}
\begin{thm}
\label{thm1}
Let $\T$ be the one--dimensional unit torus, and let $W$ satisfy the hypotheses \eqref{w1}--\eqref{w4}. Let $\alpha_0 > 0$.  Then there exist $q_0 > 0$ and $\eps_0 > 0$, possibly dependent on $\alpha_0$ and $q_0$, such that if $q < q_0$ and $w \in H^2(\T)$ has at least $N$ zeros, $\{x_{k}\}_{k=1}^{N}$, satisfying $\min_k |x_{k+1} - x_k| \ge \alpha_0$ then% and every $k$ there exists a sequence $(d_k^\e)$ satisfying $|d_k^\e -  d_k(v)| \le 2 \delta$ % := x_{k+1}(v) - x_{k}(v) 
\beq
\label{thm1res}
E_\e(w; \T) \geq N m_1 - C \sum_{k=1}^{N} \exp \left(- \frac{d_k \gamma}{\eps} \right),
\eeq
for every $0 < \eps < \eps_0$, where $d_k = x_{k+1} - x_k$,  $\gamma > 0$ is defined in \eqref{defGamma} and depends only on $W$, while $C > 0$ is independent of $\e$. % depends on the linearization of the Euler--Lagrange equation \eqref{EL} around the wells of $W$, see REFERENCE 
\end{thm}
We remark that a similar estimate can be obtained when the domain is an interval $I := (a_0, b_0)$, with \eqref{thm1res} replaced with
\beq
E_\e(w; I) \geq N m_1 - C \sum_{k=0}^{N} \exp \left(- \frac{d_k \gamma}{\eps} \right),
\eeq
where $d_0 := x_1 - a_0, \, d_{N} := b_0 - x_N$. 

\begin{rem}
We highlight the fact that we are {\it not} requiring the function $w$ of Theorem \ref{thm1} to be $L^1$--close to a jump function, in contrast with \cite{BellettiniNayamNovaga}, \cite{BronsardKohn}, \cite{Grant}, \cite{KaliesVanderVorstWanner}. On the other hand, it is easy to show that if $w$ is $L^1$--close to a jump function $v$ taking values $\pm 1$, then there exists an $\alpha_0 > 0$ with the property that the zeros of $w$ are at least $\alpha_0 > 0$ apart, as in the statement of Theorem \ref{thm1}.
\end{rem}

The energy estimate above is a crucial ingredient to prove slow motion of solutions of \eqref{PDE}, when the initial data is close in the $L^1$ norm to a $BV$ function, as in \cite{BronsardKohn}, \cite{Grant}, \cite{KaliesVanderVorstWanner}. In particular, we will consider regular solutions of \eqref{PDE}, whose existence is proved in the Appendix, see Theorem \ref{ExistenceTheorem}. Our analysis yields the following result.

\begin{thm} \label{exponentialslowmotion}
Let $v \in BV(\T; \{ \pm 1 \})$ be a function with $N(v) \neq 0$ jumps at $x_k(v)$, for $k = 1, \ldots, N(v)$, and let $q_0 > 0$ be as in Theorem \ref{thm1}.  Let $d := \min_{k} |x_{k+1}(v) - x_k(v)|$. Then there exist  $\eps_0,\delta_0 > 0$ with $d - 4\delta_0 > 0$ such that, if $u^\e$ is a solution of \eqref{PDE} with $u^\e \in L^2((0,\infty); H^4(\T))$, $u^\e_t \in L^2((0,\infty); H^2(\T))$
and initial data $u_{0,\e} \in H^2(\T)$ satisfying
\beq \label{h1}
|| u_{0,\e} - v ||_{L^1(\T)} \leq \delta
\eeq
for $0 < \delta < \delta_0$ and
\beq \label{h2}
E_{\eps}(u_0; \T) \leq E_0(v; \T) + \frac{1}{h(\eps)},
\eeq
for all $0 < \e < \e_0$ and for some function $h: (0,\infty) \to (0,\infty)$, then for all $q < q_0$,
\beq \label{31}
\lim_{\eps \to 0^+} \left\{ \sup_{0 \leq t \leq T_\e } \int_\T  | u^\e(x,t)  - u_{0,\e}(x)  | dx    \right\} = 0,
\eeq
where
\[
T_\e :=  \delta^2 \min\{ h(\eps), \exp((d - 4 \delta) \gamma/\eps)  \}.
\]
\end{thm} 

\begin{rem}
If $h(\e) = \exp(d  \gamma/ \eps)$, then
\[
T_\e = \delta \exp((d - 4\delta)  \gamma/ \eps)
\]
which is consistent with the estimates obtained in \cite{Grant} and \cite{KaliesVanderVorstWanner}. On the other hand, we remark that our Theorem \ref{exponentialslowmotion} provides more general results.
\end{rem}

\begin{rem}
To the best of our knowledge, only recently some regularity results for the Swift--Hohenberg equation have been proved, see  \cite{Giorgini}. In the statement of Theorem \ref{exponentialslowmotion} we \emph{assume} that the solutions are sufficiently regular.  In the Appendix we prove existence of solutions (though with weaker regularity) using De Giorgi's technique of Minimizing Movements (see Theorem \ref{ExistenceTheorem}).
\end{rem}

\subsection{Outline of the Proof}

A key step in proving the energy inequality \eqref{thm1res} is a bound from below by the energy of an appropriately chosen test function.  Given $w \in H^2(\T)$ satisfying the assumptions of Theorem \ref{thm1}, we follow \cite{BellettiniNayamNovaga} to construct this test function by gluing together minimizers of the energy on each subinterval $I_k := [x_k, x_{k+1}]$, where the admissible class now consists of $H^2(I_k)$ functions that equal zero at the endpoints of $I_k$.
Thus,
\beq
E_\e(w; \T) \ge \sum_k E_\e(\hat{w}_k; I_k),
\eeq 
where $\hat{w}_k$ also solves a fourth order Euler--Lagrange equation corresponding to the energy functional.  

This initial energy inequality has several key advantages.  First, it assumes no assumptions about closeness of $w$ to a step function taking values $\pm 1$.  The required estimates can be \emph{proved} for $\hat{w}_k$.  Secondly, the additional property that $\hat{w}_k$ solves a fourth order ODE on the whole subinterval is key in obtaining a sharper lower bound than the one established in \cite{KaliesVanderVorstWanner}.  Specifically, in the middle of each subinterval $I_k$, we can show that the minimizer $\hat{w}_k = \pm 1 + O(\exp \left(-\gamma (x_{k+1}-x_k)/2\e\right))$, where the exponent $\gamma$ is related to the linearization of the Euler--Lagrange equation.  In fact, obtaining this bound is the central contribution of this paper, starting from Corollary \ref{inter} and culminating in Proposition \ref{A3}.  The proofs of Lemmas \ref{2.2} and \ref{A2}, which give the initial crude estimates on the `closeness' of $\hat{w}_k$ to $\pm 1$, follow the ideas of \cite{KaliesVanderVorstWanner} supplemented by the use of the interpolation inequality given in Lemma \ref{interpolLemma} and the use of $\hat{w}_k$ instead of the original function $w$. A point of departure is Lemma \ref{LemmaODEclose}, in which the use of a Hartman--Grobman type theorem (see Theorem \ref{sell}, from \cite{Sell}), combined with the extra information on $\hat{w}_k$ and the analysis of the linearized problem, allow us to obtain sharper exponential decay estimate.

Once these bounds on $\hat{w}_k$ are obtained, we show that its energy is larger than the energy of the `optimal profile' connecting the zeros of $\hat{w}_k$ with $\pm 1$ and having energy $m_1/2$.  This is accomplished in the proof of Theorem \ref{thm1}.

In the remainder of the paper, we use the energy lower bound to obtain slow motion results in Section 3.  Finally, in the Appendix we present a proof of existence of solutions for equation \eqref{PDE} in the more general case of a bounded domain $\Omega \subset \R^n$, along with partial regularity results for the solutions themselves.

\section{Preliminaries and Assumptions}
Throughout this paper we will work with a double--well potential $W: \R \to [0,\infty)$ satisfying
\begin{align}
\label{w1} &W \in C^5(\mathbb{R}), \ W(s) = W(-s), \text{ for all } s \in \R; \\
\label{w2} &W(s) > 0, \text{ for } s \geq 0, s \neq 1; \\
\label{w3} &W(1) = W'(1) = 0; \\
\label{w4} &\text{there exists } 0 < c_W \leq 1 \text{ such that }
W(s) \ge c_W |s - 1|^2, \mbox{ for } s \ge 0.
\end{align}
A prototype for $W$ is given by 
\beq \label{particularW}
W(s) := \frac{1}{4} (s^2 - 1)^{2}.
\eeq
%and we shall work with this specific $W$, in order to make the proofs more transparent and being straightforward to generalize our results to the more general case.

\subsection{$\Gamma$--convergence and Interpolation Inequalities}
In this section we recall some properties of the energy
\beq \label{Giovanni1}
E_{\eps}(u;\Omega) := \int_{\Omega} \left( \frac{1}{\eps} W(u) - \eps q |\nabla u|^2 + \eps^3 |\nabla^2 u|^2  \right) dx, % =: \int_{\Omega} I_{\eps}[u] dx,
\eeq
in the more general setting where $\Omega$ is a bounded open set of $\R^n$ with $C^1$ boundary, $q > 0$ is a small parameter, and $W$ is a double--well potential, as in \eqref{particularW}.
In \cite{ChermisiDMFL} Chermisi, Dal Maso, Fonseca and Leoni proved that the sequence of functionals $\E_{\eps}: L^2(\Omega) \to \R \cup \{ +\infty  \}$, defined by
\[
\E_{\eps}(u) :=
\begin{cases}
E_{\eps}(u;\Omega) &\text{if} \ u \in H^2(\Omega), \\
+\infty &\text{if} \ u \in L^2(\Omega) \setminus H^2(\Omega),
\end{cases}
\]
$\Gamma$--converges as $\eps \to 0^+$ to the functional $\E_0: L^2(\Omega) \to \R \cup \{ +\infty  \}$,
\[
\E_0(u) := 
\begin{cases}
m_n \text{Per}_{\Omega}(\{ u=1  \}) &\text{if} \ u \in BV(\Omega; \{ -1, +1 \}), \\
+\infty &\text{if} \ u \in L^2(\Omega) \setminus BV(\Omega; \{ -1, +1 \}),
\end{cases}
\]
where
\[
m_n := \inf \left\{ E_{\eps}(u;Q) : 0 < \eps \leq 1, u \in  \mathcal{A}_n  \right\},
\]
$Q := \left( -\frac{1}{2}, \frac{1}{2}  \right)^n$, and
\[
\begin{aligned}
\mathcal{A}_n := \Big\{ u \in &H^2_{\text{loc}}(\R^n), \ u(x) = -1 \ \text{near} \ x \cdot e_n = - \frac{1}{2}, \\
&u(x) = 1 \ \text{near} \ x \cdot e_n = \frac{1}{2}, \\
&u(x) = u(x + e_i) \ \text{for all} \ x \in \R^n, i= 1, \ldots, n-1 \Big\}.
\end{aligned}
\]
We define the one--dimensional rescaled energy
\beq
\label{1DF}
E(v; A) := \int_{A} \left( W(v) -  q (v')^2 +  (v'')^2  \right) dx,
\eeq
and we introduce the set of admissible functions
\beq
\begin{aligned}
\mathcal{A} &:= \left\{ v \in H^2_{\text{loc}}(\R) : v(x) = -1 \ \text{near} \ x= a, v(x) = 1 \ \text{near} \ x= b    \right\}.
\end{aligned}
\eeq
We note that it was proved in \cite{ChermisiDMFL}, Section 5.1, that
\beq \label{m1}
m_1 = \inf \left \{ E(v; \R)   : v \in H^2_{\text{loc}}(\R), \lim_{x \to \pm \infty} v(x) = \pm 1  \right \},
\eeq
so that in dimension $n = 1$ we have
\[
\E_0(u) =
\begin{cases}
Nm_1 &\text{if} \ u \in BV((a,b); \{ -1, +1 \}) \\
+\infty &\text{if} \ u \in L^2((a,b)) \setminus BV((a,b); \{ -1, +1 \}),
\end{cases}
\]
where $N$ is the number of jumps of the function $u$.
We further define
\begin{align}
\label{mpmDef}
m_\pm & := \inf \big \{ E(u; \R^+)   : u \in H^2_{\text{loc}}(\R^+), \lim_{x \to \infty} u(x) = \pm 1, \quad u(0) = 0  \big \} \nonumber \\
& = \inf \big \{ E(u; \R^-)   : u \in H^2_{\text{loc}}(\R^-), \lim_{x \to -\infty} u(x) = \pm 1, \quad u(0) = 0  \big \}
\end{align}
and remark that in our case of symmetric potential $W$, $m_+ = m_- = m_1/2$. One of the key tools to prove the $\Gamma$--convergence result is the following nonlinear interpolation inequality, see e.g. Theorem 3.4 in \cite{ChermisiDMFL}.

%In particular, the proof was given in \cite{fl} in dimension $n \geq 1$ and in \cite{csz} for $n=1$. We recall the result for our potential and we extend the proof to the case of a more general potential, where the assumption \eqref{w4} is replaced with
%\beq \tag{w4'}
%\label{w4prime} \text{There exists } c_W > 0 \text{ such that }
%W(s) \ge c_W |s^2 - 1|^p, \mbox{ for } 1 \le p < 2. 
%\eeq
%We now prove an a analogous result, in the case $1 \leq p < 2$. 
%The proof is very similar to the one of Lemma 3.1 in \cite{csz} and we shall recall the identical parts for the convenience of the reader.
\begin{lemma} \label{interpN}
Let $\Omega$ be a bounded open set of $\R^n$ with $C^1$ boundary, and assume that $W$ satisfies \eqref{w1}--\eqref{w4}. Then there exists a constant $q^* > 0$, independent of $\Omega$, such that for every $- \infty < q < q^*/N$ there exists $\e_0 = \e_0(\Omega, q) > 0$ such that
\[
q\e^2 \int_\Omega |\nabla u|^2 dx \leq \int_\Omega W(u) dx + \e^4  \int_\Omega |\nabla^2 u|^2 dx
\]
for every $\e \in (0, \e_0)$ and every $u \in H^2(\Omega)$.
\end{lemma}

In particular, in the one dimensional setting, we will often use the following nonrescaled version of the previous result, see Lemma 3.1 in \cite{CicaleseSpadaroZeppieri}.

\begin{lemma} \label{interpolLemma}
Let $W$ be a continuous potential satisfying \eqref{w2}--\eqref{w4}. Let $I \subset \R$ be an open, bounded interval. Then there exists a constant $q^* > 0$ such that 
\[
\label{EQ}
q^* \int_I (u')^2 dx \leq \frac{1}{\La^1(I)^2} \int_I W(u) dx + \La^1(I)^2\int_I (u'')^2 dx
\]
for every $u \in H^2(I)$.
\end{lemma}

\begin{corollary} \label{inter}
Let $W$ and $q^*$ be as in Lemma \ref{interpolLemma}.  Then there exist $\sigma > 0$ such that for every open interval $I$, every $0 < \eps \leq \La^1(I)$, and every $-\infty < q \leq q^*/4$ ,
\beq
\label{inter1}
q \eps^2 \int_I (u')^2 dx \leq \int_I \left( W(u) + \eps^4 (u'')^2  \right) dx,
\eeq
and
\beq
\label{inter2}
E_\eps(u; I) \ge \sigma \int_I  \left( W(u) + \eps^2 (u')^2 + \eps^4 (u'')^2 \right) dx.
\eeq
for all $u \in H_{loc}^2(I)$.
\end{corollary}
\begin{proof}
Let $I = (a,b)$ and $u \in H^2((a,b))$.  We change variables $v(y) := u(\eps x)$, 
subdivide the resulting rescaled domain $I_\eps = (a/\eps, b/\eps)$ into $\left[\frac{b-a}{\eps}\right]$ subintervals, $I_\eps^k$, of length between $1/2$ and $2$ (since $0 < \eps \le b-a$) and use Lemma \ref{interpolLemma} to obtain
\begin{align}
\label{interPol3}
\frac{q^*}{4}  \int_{a}^{b} (u')^2 dx & = \frac{q^*}{4 \eps}  \int_{a/\eps}^{b/\eps} (v')^2 dy =  \frac{1}{4\eps} \sum_k q^* \int_{I_\eps^k} (v')^2 dy \leq  \frac{1}{4\eps} \sum_k  \int_{I_\eps^k}  \left( 4 W(v) + 4 (v'')^2  \right) dy \nonumber \\
&  = \frac{1}{\eps}  \int_{a/\eps}^{b/\eps}  \left( W(v) + (v'')^2  \right) dy  = \int_a^b (W(u) + \eps^3 (u'')^2 ) dx .
\end{align}
Since $q \le q^*/4$, \eqref{inter1} easily follows.  To prove \eqref{inter2} we follow closely the strategy used in the proof of Theorem 1.1 of \cite{ChermisiDMFL} and proceed as follows. Fix $\sigma \in (0,1)$  sufficiently small so that $(q + \sigma)/ (1- \sigma) < q^*/4$.  Then,
\beq
\label{enSp1}
\begin{aligned}
\int_a^b (W(u) - q\eps^2 (u')^2 + \eps^4 (u'')^2) dx &= (1-\sigma) \int_a^b \left( W(u) - \frac{q + \sigma}{1- \sigma} \eps^2 (u')^2 + \eps^4 (u'')^2 \right) dx \\
&+ \sigma \int_a^b \left( W(u) + \eps^2 (u')^2 + \eps^4 (u'')^2 \right) dx,
\end{aligned}
\eeq
and \eqref{inter2} follows since by \eqref{interPol3} the first term on the right-hand side of \eqref{enSp1} is nonnegative.

\end{proof}

%\textcolor{blue}{NEW
%\begin{corollary} 
%There exists $q^* > 0$ such that for every open interval $J \subset \R$, every $0 < \eps \leq |J|$, and every $-\infty < q \leq q^*/4$ ,
%\[
%q \eps^2 \int_J (u')^2 dx \leq \int_J \left[ W(u) + \eps^4 (u'')^2  \right] dx,
%\]
%for all $u \in H_{loc}^2(J)$.
%\end{corollary}
%\begin{proof}
%We follow the same line of proof of Proposition 3.3 in \cite{csz}. Setting $v(x) := u(\eps x)$ and for $J = (a_1,a_2)$, we see that it is enough to prove
%\[
%\label{interPol3}
%\frac{q^*}{4}  \int_{J_\eps} (v')^2 dx \leq \int_{J_\eps}  \left[ W(v) + (v'')^2  \right] dx,
%\]
%where $J_\eps =  (a_1/\eps, a_2/\eps)$. We now subdivide the rescaled domain into $J_\eps$ into $n_\eps := \left[\frac{a_2-a_1}{\eps}\right]$ subintervals $J^i_\eps$ and notice that their length is between $1/2$ and $2$, since $0 < \eps \le |a_2 - a_1|$. Applying Lemma \ref{interpolLemma} to each interval $J^i_\eps$ gives
%\[
%\begin{aligned}
%E_\eps[u; I] &= \sum_{i=1}^{n_\eps} \int_{J^i_\eps} W(u) - k (u')^2 + (u'')^2 dx \\
%&\geq \left( 1- \frac{k}{4q^*} \right) \int_{J_\eps} W(u) dx + \left( 1- \frac{k}{4q^*} \right) \int_{J_\eps} (u'')^2 dx,
%\end{aligned}
%\]
%and the desired result now follows.
%\end{proof}
%}
The following lemmas established for a generalization of the  Modica--Mortola Functional in \cite{KaliesVanderVorstWanner} will be useful to prove our main result.  While our energy does not satisfy the assumptions of \cite{KaliesVanderVorstWanner}, their argument is easily extended to our case with the help of the  interpolation inequality \eqref{inter2}.  In particular, Lemma \ref{2.2},  shows that an $H^2$ function with a uniformly bounded energy, necessarily takes values close to $\{\pm 1\}$ and has small derivatives, except on a set of measure $O(\eps)$ and Lemma \ref{A2} gives a characterization of the global minimizers for the energy $E(\cdot, \cdot)$, defined in \eqref{1DF}, subject to small boundary conditions. 
\begin{lemma} \label{2.2}
Let $I$ be an open interval, $M > 0$ and $0 < \delta < 1$. Then there exists a constant $C_1 > 0$ such that  for any $0 < \eps \leq \La^1(I)$
%it was  (b-a)/d_0, d_0 was coming from their H1.
and every $u \in H^2(I)$ with $E_{\eps}(u; I) \leq M$ the following property holds: there is a measurable set $J \subset I$ with $\La^1(J) \leq C_1 \eps$ such that
\[
{\rm dist}(u(x), \{ \pm 1 \}) < \delta \quad and \quad |\eps u'(x) | < \delta \quad  \text{and}
\]
hold for all $x \in I \setminus J$, where ${\rm dist}$ denotes the usual distance between a point and a set.
\end{lemma}
\begin{proof}
By \eqref{inter2}, for every $0 < \eps \leq \La^1(I)$ and $u \in H^2(I)$,
\beq \label{pot}
\begin{aligned}
\int_I \left( W(u) - q\eps^2 (u')^2 + \eps^4 (u'')^2 \right) dx &\geq \sigma \int_I \left( W(u) + \eps^2 (u')^2 + \eps^4 (u'')^2 \right) dx \\
&\geq \sigma \int_I W(u) dx.
\end{aligned}
\eeq
We now let $J_0 := \{ x \in I: \text{dist}(u(x), \{ \pm 1 \} ) \geq \delta   \}$ and from the definition of $W$ we have $c := \inf \{  W(s): \text{dist}(s, \{ \pm 1 \} ) \geq \delta  \} > 0$. Then \eqref{pot} implies
\[
M \geq E_{\eps}(u; I) \geq \frac{\sigma}{\eps} \int_I W(u) dx \geq \frac{c \sigma}{\eps} \La^1(J_0),
\]
and therefore 
\[
\La^1(J_0) \leq \frac{M \eps}{c \sigma}.
\]
Similarly, setting $J_1 := \{ x \in I: | \eps u'(x) | \geq \delta  \}$,  \eqref{pot} yields the estimates
\beq \label{pot2}
M \geq E_{\eps}(u;I) \geq \frac{\sigma}{\eps} \int_I \eps^2 (u')^2 dx \geq \frac{\sigma \delta^2}{\eps} \La^1(J_1)
\eeq
and consequently
\[
\La^1(J_1) \leq \frac{M \eps}{\sigma \delta^2}.
\]
%The bound for $|\eps^2 u''(x) |$ on the corresponding set $J_2$ follows analogously. 
Setting $J := J_0 \cup J_1$ yields the desired result.
\end{proof}
%We conclude this section with characterizing the global minimizers for the energy $E[\cdot]$ subject to small boundary conditions. 
\begin{lemma} \label{A2}
Let $I := (a,b)$ be an open interval and $W \in C^2$ satisfy \eqref{w2}--\eqref{w4}. Given $\alpha = (\alpha_1,  \alpha_2) \in \mathbb{R}^2$, $\beta = (\beta_1, \beta_2 ) \in \mathbb{R}^2$ define
\beq
\mathcal{M}^\pm_{\alpha,\beta} := \{v \in H^2(I): v(a) = \pm 1 + \alpha_0, \, v'(a) = \alpha_1, \, v(b) = \pm 1 + \beta_0, \, v'(b) = \beta_1 \}.
\eeq
Then there exist constants $\delta_0, C > 0$ such that the following holds.  If $\La^1(I) > 1$ and $||\alpha||, ||\beta|| \le \delta < \delta_0$ then the functional $E(\cdot; I)$ defined in \eqref{1DF} has a global minimizer $v_\pm$ on $\mathcal{M}^\pm_{\alpha, \beta}$.  This minimizer $v_\pm$ solves the Euler--Lagrange equation, and satisfies the estimates 
\beq \label{A2bound1}
||v_\pm  \pm 1||_{L^\infty(I)} \le C \delta,
\eeq
\beq \label{A2bound2}
||v^{(k)}_\pm|_{L^{2}(I)} \le C \delta \text{ for } k = 1, \dots, 4.
\eeq
\beq \label{A2bound3}
||v^{(k)}_\pm||_{L^{\infty}(I)} \le C \delta \text{ for } k = 1, \dots, 3.
\eeq
\end{lemma}

\begin{figure}[h]
\centering{
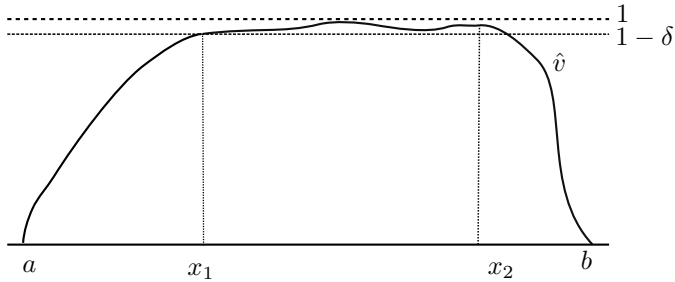 %\resizebox{75mm}{!}{}
\caption{If $\hat v$ is close to $1$ at $x_1$ and $x_2$, then it stays close in between.}
\label{fig1}
}
\end{figure}

\begin{proof}
We prove the proposition when $s = -1$, the $s=1$ case being identical.
We divide the proof into several steps. Moreover, we simplify the notation used for the $L^p$ norms when the domain of integration will be clear from the context. \\

\noindent {\bf Step 1.} Fix $\delta > 0$. We claim that there exists $C_1 > 0$ such that if $||\alpha||, ||\beta|| \le \delta$, then 
\beq
\label{infBound}
\inf_{\mathcal{M}^-_{\alpha, \beta}} E(\cdot; I) \le C_1 \delta^2.
\eeq
To show this we note that, if $\varphi_0, \varphi_1 \in C^\infty(\R)$ satisfy $\varphi_i(x) = 0$ for all $x \geq 1/2$, with $\varphi_0(0) = 1, \varphi_0'(0) = 0, \varphi_1(0) = 0$, and $\varphi_1'(0) = 1$, then the function
\beq
\phi(x) := -1+\alpha_0 \varphi_0(x-a) + \alpha_1 \varphi_1(x-a) + \beta_0 \varphi_0(b-x) - \beta_1 \varphi_1(b-x), \ x \in (a,b),
\eeq
belongs to $\mathcal{M}^-_{\alpha,\beta}$.
Using $\phi$ as a test function, \eqref{infBound} follows from Taylor's formula for $W$ and the facts that $W(\pm 1) = W'(\pm 1)= 0$ and $W \in C^2(\mathbb{R})$. \\

\noindent  {\bf Step 2.} Fix $0 < \delta < 1$. We will show that there exists $C_2 > 0$ such that for every $v \in \mathcal{M}^-_{\alpha, \beta}$, with $v \leq 0$ on $I$ and $||\alpha||, ||\beta|| \le \delta$ we have 
\beq \label{step2Ineq}
E(v; I) \ge C_2 ||v+1||_{L^\infty}^2.
\eeq

Suppose that $|v(x)+1| \ge ||v+1||_{\infty}/2$ for all $x \in I$. Using \eqref{w4} and \eqref{inter2} with $\eps = 1$ we have,
\beq
E(v; I) \ge \sigma\int_I W(v) dx \ge \sigma c_W  \int_I |v+1|^{2} \ge \La^1(I) \frac{\sigma}{4} c_W ||v+1||_{L^\infty}^2
\eeq
Otherwise, there are points $x_0, x_1 \in \bar I$ satisfying 
\[
|v(x_0)+1| = \frac{||v+1||_\infty}{2} \ \text{ and } \  |v(x_1) + 1| = ||v + 1||_\infty, 
\]
in which case, again by \eqref{w4}, \eqref{inter2} and Young's Inequality
\[
\begin{aligned}
E(v; I)  & \ge \sigma \int_I \left( W(v) + |v'|^2 \right) dx \ge 2\sigma \int_I \sqrt{W(v)} |v'| \\
&\ge 2 c_W \sigma \left| \int_{x_0}^{x_1} |v+1| v' dx \right|  \\
&= c_W \sigma  \left( (v+1)^2(x_1) - (v+1)^2(x_0) \right) = \frac{\sigma}{2} c_W ||v +1 ||_{L^\infty}
\end{aligned}
\]
and this proves \eqref{step2Ineq}. \\

\noindent {\bf Step 3.} We claim that there exists $\delta_0 >0$ and $C_3 = C_3(\delta_0) > 0$ such that if $||\alpha||, ||\beta|| \le \delta < \delta_0$ and $v \in \mathcal{M}^-_{\alpha,\beta}$, with $E(v; I) \le 2 \inf_{\mathcal{M}^-_{\alpha,\beta}} E$, then 
\beq \label{step3Ineq}
||v+1||_{L^\infty} \le C_3 \delta.
\eeq

By taking $0 < \delta < 1$ sufficiently small, we may assume that $v \le 0$ on  $I$.  Indeed, since $v(a) = -1 + \alpha_0 \leq -1 + \delta < 0$, if $v(x) > 0$ for some $x$, then necessarily there exists $x_1$ such that $v(x_1) = 0$, and so by \eqref{w4},
\[ 
\begin{aligned}
E(v; I) & \ge \sigma \int_I \left( W(v) + |v'|^2 \right) dx \ge 2\sigma \int_a^{x_1} \sqrt{W(v)} |v'| \\
&\ge 2 C_W \sigma \left| \int_{a}^{x_1} |v+1| v' dx \right| \\
&= \sigma C_W \left( (v+1)^{2}(x_1) - (v+1)^{2}(a) \right) \\
&\ge \sigma  (1 - |\alpha_0|^{2}),
\end{aligned}
\]
which contradicts Step 1 for $\delta$ sufficiently small. Hence, Steps 1 and 2 imply \eqref{step3Ineq}. \\

\noindent {\bf Step 4.} Finally, \eqref{inter2} with $\e = 1$ and standard compactness and lower semicontinuity arguments imply the existence of minimizer $v_-$ of $E(\cdot; I)$ and since by previous step $v_- \leq 0$ for $\delta < \delta_0$ and
\beq \label{vBounded}
||v_-  +1||_{L^2}^{2} \le \La^1(I) ||v_- +1||_{L^\infty}^{2} \leq C \delta^2,
\eeq
for some $C > 0$, again using \eqref{inter2} along with \eqref{infBound} yields
\[
||v_-^{(k)} ||_{L^2} \leq C\delta, \text{ for } k = 1,2.
\]
Furthermore, since $W$ is $C^2$, from \eqref{vBounded} and the Mean Value Theorem we have
\beq \label{MVT1}
W'(v_-) = W'(v_-) - W'(-1) \leq \max_{0 \leq \xi \leq 2} W''(\xi)  (v_- +1).
\eeq
The Euler--Lagrange equation
\[
2 v_-^{(iv)} + 2q v_-'' + W'(v_-) = 0,
\]
the $L^\infty$ bound from Step 3 and \eqref{MVT1} imply
\beq
||v_-^{(iv)}||_{L^2} \le  |q| ||v_-''||_{L^2} + \frac{1}{2} ||W'(v_-)||_{L^2} \le |q| ||v_-''||_{L^2} +  C ||v_- + 1||_{L^2} \le C \delta % \frac{1}{2} ||\hat{w}^{(4)}||_{L^2(a,b)} + C ||\hat{w}||_{L^2(a,b)}
\eeq
for some $C > 0$.

\iffalse
where we have used Taylor's formula for $W'$ and, in particular, the fact that
\[
|W'(v_-)| \leq \left( \max_{[-1 - c\delta, -1+c\delta ]} |W''| \right) |v_- + 1|
\]
for some constant $c > 0$. 
\fi

The energy bound \eqref{infBound} and standard interpolation inequalities (e.g., see Theorem 6.4 in \cite{FFLM}) imply \eqref{A2bound1}, \eqref{A2bound2}, \eqref{A2bound3}.
\end{proof}

\subsection{The Euler--Lagrange Equation}
In this section we further analyze the behavior of the minimizers of the energy $E_\eps$ with the aid of the corresponding Euler-Lagrange equation, and we prove our main result, Theorem \ref{thm1}.

\begin{lemma} \label{LemmaODEclose}
Consider the ordinary differential equation 
\beq
\label{odeMain}
x' = F(x),
\eeq
where $F: \R^4 \to \R^4$ is a $C^4$ mapping satisfying $F(x_0) = 0$ for some $x_0 \in \R^4$.
Assume $DF(x_0)$ has four eigenvalues $\pm \gamma \pm \delta i$, where $\gamma > 0$ and $\delta \in \mathbb{R}$.
Then for $0 < \la \leq \gamma$ there exist a constant $C(\gamma, \delta) > 0$, $T_0(\gamma, \delta) > 0$ and $R > 0$ such that for all $T>T_0$, if $x : [0,T] \rightarrow B(x_0, R)$ is a solution of \eqref{odeMain}, then the inequality
\beq \label{bound}
|x(t) - x_0| \le C(\gamma,\delta) \exp \left(-\lambda T /2\right)
\eeq
holds for all $t \in \left[ \frac{\la T}{2 \gamma}, T - \frac{\la T}{2 \gamma}  \right]$. In particular, if $\gamma = \lambda$,
\beq \label{boundMiddle}
|x(T/2) - x_0| \le C(\gamma, \delta) \exp \left(-\gamma  T /2\right).
\eeq
\end{lemma}
\begin{proof}
Changing variables if necessary, we may assume, without loss of generality, that $x_0 = 0$.  Let $A := DF(0)$.
By an extension of the Hartman--Grobman Theorem (see, e.g. \cite{Sell} and Lemma \ref{sell2} in the Appendix),
there exist two open neighborhoods of $0$, $V_1, V_2 \subset \R^4$, and a diffeomorphism $h: V_1 \rightarrow V_2$ of class $C^1$, with $h(0) = 0$, such that if $x(t) \in V_1$ for all $t \in [0,T]$ then the funciton $y(t) := h(x(t)), t \in [0,T]$ is a solution of the linearized system 
\beq
\label{ode1}
y' = A y.
\eeq
Let $R > 0$ be so small that $\overline{B(0,R)} \subset V_1$, and define $V:=h(B(0,R))$.  Then $V$ is bounded and since $h(0) = 0$, there exists $L > 0$ such that $V \subset B(0,L)$.
Hence if $x(t) \in B(0,R)$ for all $t \in [0,T]$, then $y(t) \in B(0,L)$ for all $t \in [0,T]$.

%In fact, since the eingenvalues of $A$ are assumed to all be distinct, we have $A = S^{-1} D S$, where $D$ is a diagonal matrix consisitng of eigenvalues of $A$.  Hence, considering $y \rightarrow Sy$ as another diffeomorphism of a neighborhood of $0$ we may assume without loss of generality that $A$ is diagonal.
Since the eigenvalues of $A$ are all distinct, the solution of \eqref{ode1} has the form
\[
y(t) = c_1 v_1 \exp \left((-\gamma-\delta i)t\right) + c_2 v_2 \exp \left((-\gamma+\delta i) t\right) + c_3 v_3 \exp \left((\gamma-\delta i) t\right) + c_4 v_4 \exp \left((-\gamma-\delta i) t\right),
\]
where $c_1, \dots, c_4$ are complex valued constants and $\{v_i\} \subset \C^4$ is a linearly independent set of eigenvectors of $A$. Letting $P = [v_1, v_2, v_3, v_4]$ be the matrix of eigenvectors of $A$, we write the above solution as
\beq
\label{rbound2}
y(t) = P [c_1 \exp \left((-\gamma-\delta i)t\right), c_2 \exp \left((-\gamma+\delta i)t\right), c_3 \exp \left((\gamma-\delta i)t\right), c_4 \exp \left((\gamma-\delta i)t\right)]^{\rm{Tr}},
\eeq
where the superscript $\rm Tr$ denotes the transpose of a matrix. Since $y(t) \in B(0,L)$ for all $t \in [0,T]$,
\[
\begin{aligned}
| [c_1 \exp \left((-\gamma-\delta i)t\right), c_2 \exp \left((-\gamma+\delta i)t\right), c_3 \exp \left((\gamma-\delta i)t\right), c_4 \exp \left((\gamma-\delta i)t\right)] |^2 &\le ||P^{-1}||^2 |y(t)|^2 \\
&\le L^2 ||P^{-1}||^2,
\end{aligned}
\]
where $||P^{-1}||$ is the operator norm of $P^{-1}$.
In particular,
\beq %\label{aichoice}
|c_1|^2  \le L^2 ||P^{-1}||^2 \exp \left(2\gamma t\right), \quad |c_2|^2 \le L^2 ||P^{-1}||^2 \exp \left(2\gamma t\right), 
\eeq
\beq %\label{aichoice2}
|c_3|^2 \le L^2 ||P^{-1}||^2 \exp \left(-2\gamma t \right),  \quad |c_4|^2 \le L^2 ||P^{-1}||^2 \exp \left(-2\gamma t \right),
\eeq
for all $t \in [0,T]$.
Setting $t = 0$ and $t = T$ in the first and second row respectively we obtain bounds on the constants $c_1, ..., c_4$,
\beq \label{aichoice}
|c_1|  \le L ||P^{-1}||, \quad |c_2| \le L ||P^{-1}||, 
\eeq
\beq \label{aichoice2}
|c_3| \le L ||P^{-1}|| \exp \left(-\gamma T \right),  \quad |c_4| \le L ||P^{-1}|| \exp \left(-\gamma T \right).
\eeq
Using the resulting bounds in \eqref{rbound2} yields 
\[
\begin{aligned}
\exp \left(\lambda T \right) |y(t)|^2 &\le \exp \left(\lambda T \right) ||P||^2 \left(  |c_1|^2 \exp \left(-2\gamma t \right) + |c_2|^2 \exp \left(-2\gamma t \right) + |c_3|^2 \exp \left(2\gamma t \right) + |c_4|^2 \exp \left(2\gamma t \right) \right) \\
&\le 4 L^2 ||P||^2 ||P^{-1}||^2,
\end{aligned}
\]
provided 
\[
\lambda T - 2 \gamma t \le 0 \ \text{ and } \ \lambda T - 2\gamma T + 2\gamma t \le 0.
\]
Both of these conditions are satisfied as long as 
\[
t \in \left[ \frac{\la T}{2 \gamma}, T - \frac{\la T}{2 \gamma}  \right] =: [t_1,t_2].
\]
Hence for $t \in [t_1, t_2]$,
\[
|y(t)|^2 \le 4 L^2 ||P||^2 ||P^{-1}||^2 \exp \left(-\lambda T \right).
\]
In particular, if $T$ is sufficiently large (depending only on $\gamma, \delta$, and $V_2$), there exists a compact set $E$ such that $y(t) \in E \subset V_2$ for all $t \in [t_1, t_2]$.
Since $h^{-1}$ is $C^1$ and $h(0) = 0$, by the Mean Value Theorem,
\beq \label{finaleq}
|x(t)| =  |h^{-1} (y(t))| \le \sup_{s \in E} |\nabla h^{-1}(s)| |y(t)| \le C_{\gamma,\delta} \exp \left(-\lambda T/2 \right)
\eeq
for all $t \in [t_1, t_2]$, where $C_{\gamma,\delta} := L  \sup_{s \in E} |\nabla h^{-1}(s)|  ||P|| ||P^{-1}||$.
\end{proof}

For a given open interval $I$ and a subinterval $(y_1, y_2) \subset I$ we define
\beq
\mathcal{M} := \left\{ w \in H^2((y_1, y_2)): w(y_1) = 0, \,  \, w(y_2)= 0 \right\}.
\eeq

\begin{prop} \label{A3} 
%Let $\{v_\eps\} \in H^2(a,b)$ be a sequence converging in $L^1(a,b)$ 
%Let $v \in BV((a,b), \{\pm 1\}\}$ with $N > 0$ jumps at $x_1, \dots, x_N$. 
Let $\eps_0 > 0$ and let $\hat{w}_\e$ be a global minimizer of $E_\eps(\cdot; (y_1,y_2))$  on $\mathcal{M}$ satisfying 
\beq
\label{hatWUB}
E_\eps(\hat{w}_\e; (y_1,y_2)) \le M,
\eeq
for all $\eps < \eps_0$.
Then $\hat{w}_\e$ solves the Euler--Lagrange equation 
\beq \label{ELEp}
2\eps^4 \hat{w}_\e^{(iv)} + 2q \eps^2 \hat{w}_\e'' + W'(\hat{w}_\e) = 0,
\eeq
with additional natural boundary conditions $\hat w_\e''(y_1) = \hat w_\e''(y_2) = 0$, and for all $\eps < \eps_0$ satisfies the estimates
\beq
\label{A3Res1}
{\rm dist}(  \hat{w}_\e((y_1 + y_2)/2) , \{\pm 1\}) \le C_M \exp \left(-{\frac{d \gamma}{2\eps}} \right),
\eeq
\beq
\label{A3Res2}
|  \hat{w}_\e^{(m)} ((y_1 + y_2)/2)  | \le C_M \exp \left(-{\frac{d \gamma}{2\eps}} \right), \quad m = 1, \dots 3,
\eeq
where $d := y_2 - y_1$ and $C_M > 0$ is a positive constant dependent only on $M, q$ and the potential $W$.
%\beq
%||D^k \hat{w}||_{L^{2}(a,b)} \le C_b \delta^{2p(p-1)/(p+2)}, \text{ for } k = 1, \dots, 4.
%\eeq
%\beq
%||D^k \hat{w}||_{L^{\infty}(a,b)} \le C_b \delta^{2p(p-1)/(p+2)}, \text{ for } k = 1, \dots, 3.
%\eeq
\end{prop}

\begin{figure}[h]
%\centering
%{
\raggedleft %replace with \centering{} to center.
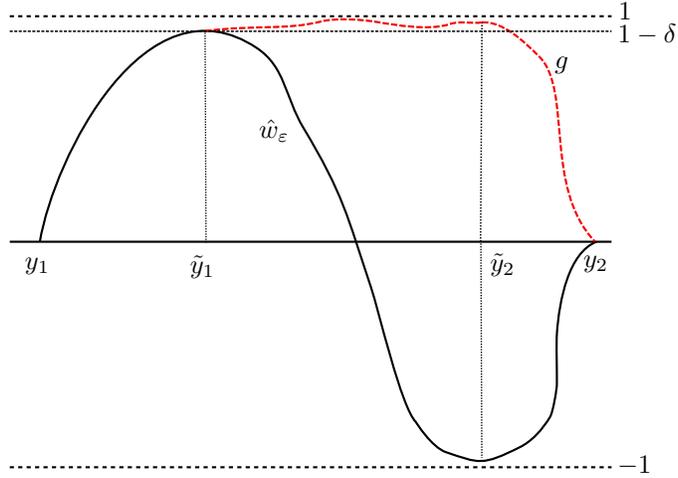 %\resizebox{75mm}{!}{}
\caption{The contradiction argument.}
\label{fig2}
%}
\end{figure}

\begin{proof}
Fix $\delta > 0$ to be chosen later. 
We first observe that, due to the upper bound \eqref{hatWUB} and Lemma \ref{2.2}, there exists $c = c(\delta, M) > 0$ and  points $\tilde{y}_1 \in (y_1, y_1 + c\eps)$ and $\tilde{y}_2 \in (y_2 - c \eps, y_2)$ such that
\beq
\label{ux1}
\mbox{dist} (\hat{w}_\e(\tilde{y}_1) , \{\pm 1\}) < \delta, \quad |\eps \hat{w}'_\e(\tilde{y}_1)| < \delta,
\eeq
\beq
\label{ux2}
\mbox{dist}( \hat{w}_\e(\tilde{y}_2), \{ \pm 1 \}) < \delta, \quad |\eps \hat{w}'_\e(\tilde{y}_2)| < \delta.
\eeq
In addition, we claim that since $\hat{w}_\e$ is a minimizer, at $\tilde{y}_1$ and $\tilde{y}_2$ its value is near the same well of $W$, i.e., we may assume without loss of generality that
\beq
\label{ux3}
|\hat{w}_\e(\tilde{y}_1) - 1| < \delta, \quad |\hat{w}_\e(\tilde{y}_2) - 1| < \delta.
\eeq
As a matter of fact, if this was not the case and for example
\beq
\label{ux4}
|\hat{w}_\e(\tilde{y}_1) - 1| < \delta, \quad |\hat{w}_\e(\tilde{y}_2) + 1| < \delta.
\eeq
then consider
\[
g(x) := 
\begin{cases}
\hat{w}_\e(x), &y_1 \leq x \leq \tilde{y}_1, \\
\phi(x), &\tilde{y}_1 \leq x \leq \tilde{y}_2, \\
-\hat{w}_\e(x), &\tilde{y}_2 \leq x \leq y_2,
\end{cases}
\]
where
\beq \label{phiCompute}
\begin{aligned}
\phi(x) :=  1 &+ ( \hat{w}_\e(\tilde{y}_1) -1 )\vp_0(x - \tilde{y}_1) + \hat{w}'_\e(\tilde{y}_1)\vp_1(x - \tilde{y}_1)  \\
&+  ( -\hat{w}_\e(\tilde{y}_2) -1 )\vp_0(\tilde{y}_2 - x) + \hat{w}'_\e(\tilde{y}_2)\vp_1(\tilde{y}_2 - x), 
\end{aligned}
\eeq
and $\vp_0, \vp_1$ satisfy
\[
\vp_j \in C^{\infty}\left( \R \right), \ \vp_j(x) = 0 \mbox{ for all } x \geq (y_2 - y_1)/2,
\]
\[
\vp_0(0) = 1, \ \vp_0'(0) = 0, \ \vp_1(0) = 0, \ \vp_1'(0) = 1.
\]
It is easy to see that
\[
\phi(\tilde{y}_1) = \hat{w}_\e( \tilde{y}_1), \quad \phi'(\tilde{y}_1) = \hat{w}'_\e( \tilde{y}_1),
\]
\[
\phi(\tilde{y}_2) = -\hat{w}_\e( \tilde{y}_2), \quad \phi'(\tilde{y}_2) = -\hat{w}'_\e( \tilde{y}_2),
\]
and consequently $g \in H^2((y_1, y_2))$. Obtaining $\phi'$ from \eqref{phiCompute} and using \eqref{ux4}, we get
\[
||\phi'||_{L^\infty(\tilde y_1,\tilde y_2)}^2 \leq c (||\vp_0'||_{L^\infty(\R)}^2 + ||\vp_1'||_{L^\infty(\R)}^2) \delta^2, 
\]
where $c > 0$ is a constant and we notice that
\[
\int_{\tilde y_1}^{\tilde y_2} |\phi' |^2 dx \leq c (y_2 - y_1) (||\vp_0'||_{L^\infty(\R)}^2 + ||\vp_1'||_{L^\infty(\R)}^2) \delta^2.
\]
Similarly, an analogous bound for $\phi''$ can be derived. Additionally, using Taylor's formula for $W$ and the facts that $W(\pm 1) = W'(\pm 1)= 0$ and $W \in C^2(\mathbb{R})$, it follows that  
\beq
\label{enTest}
E_\e(\phi; (\tilde{y}_1, \tilde{y}_2)) \leq \xi_1 \delta^2, 
\eeq
where $\xi_1$ only depends on $y_1$ and $y_2$, which do not depend on $\delta$, while interpolation inequality of Corollary \ref{inter} yields for $\delta$ sufficiently small
\[
\begin{aligned}
 E_\e( \hat{w}_\e ; (\tilde{y}_1, \tilde{y}_2)) &= \int_{\tilde{y}_1}^{\tilde{y}_2} \left( \frac{1}{\eps}  W(\hat{w}_\e) - q \eps |\hat{w}_\e'|^2 + \eps^3 |\hat{w}_\e''|^2 \right) dx  \geq \sigma  \int_{\tilde{y}_1}^{\tilde{y}_2} \left( \frac{1}{\eps}  W(\hat{w}_\e) + \eps |\hat{w}_\e'|^2 \right) dx  \\
&\geq \sigma  \int_{\tilde{y}_1}^{\tilde{y}_2} \sqrt{W(\hat{w}_\e)} \hat{w}_\e' dx = \sigma \int_{\hat{w}_\e(\tilde{y}_1)}^{\hat{w}_\e(\tilde{y}_2)}  \sqrt{W(s)} ds  \geq  \sigma \int_{-\frac{1}{2}}^{\frac{1}{2}}  \sqrt{W(s)} ds =:  \sigma \xi_2> 0.
\end{aligned}
\]
In turn, from \eqref{enTest}, possibly choosing $\delta$ even smaller we get a contradiction with the fact that $\hat{w}_\e$ is a minimizer.

Since $\hat{w}_\e$ is a minimizer of $E_\eps(\cdot; (y_1, y_2))$, it follows from standard arguments that it satisfies the Euler--Lagrange equation \eqref{ELEp}.  
We change variables $z = \frac{x - y_1}{\eps}$ and define $\hat{v}(z) := \hat{w}_\e(x)$. Observe that
\beq
E_\eps(\hat{w}_\e; (y_1,y_2)) =  E(\hat{v} ; (0, d /\eps))
\eeq
and the rescaled minimizer $\hat{v}$ satisfies the Euler--Lagrange equation 
\beq \label{EL}
2\hat{v}^{(iv)}  + 2q\hat{v}'' + W'(\hat{v} ) = 0, \quad \hat{v}''(0) = \hat{v}''(d /\eps) = 0.
\eeq
We now apply Lemma \ref{A2} on the interval $\left( \frac{\tilde{y}_1 - y_1}{\eps},  \frac{\tilde{y}_2 - y_1}{\eps} \right)$ with
\beq
\alpha_0 := \hat{w}_\e(\tilde{y}_1) = \hat{v} \left( \frac{\tilde{y}_1 - y_1}{\eps} \right), \quad \alpha_1 := \eps \hat{w}_\e'(\tilde{y}_1) = \hat{v}' \left( \frac{\tilde{y}_1 - y_1}{\eps} \right),
\eeq
\beq
\beta_0 := \hat{w}_\e(\tilde{y}_2) = \hat{v} \left( \frac{\tilde{y}_2 - y_1}{\eps} \right), \quad \beta_1 := \eps \hat{w}_\e'(\tilde{y}_2) = \hat{v}'\left( \frac{\tilde{y}_2 - y_1}{\eps} \right).
\eeq
The resulting minimizer agrees with $\hat{v}$ on this interval and given $R > 0$, for $\delta$ sufficiently small the bounds  \eqref{ux1} and \eqref{ux2} imply that
\[  
\chi := [\hat{v}-1,\hat{v}',\hat{v}'',\hat{v}'''] \in B(0, R).
\]
\iffalse
Taylor expanding $t \mapsto W'(t)$ about $t=1$ and noting that $W''(1) = 0$ yields from \eqref{EL}
\beq \label{EL2}
2\hat{v}^{(iv)} + 2q\hat{v}''  + W''(1) (\hat{v} - 1) +  g(\hat{v}-1) = 0,
\eeq
where we have introduced the approximation error 
\beq \label{approxError}
g(s) := W'(s+1) - W'(1) - W''(1)s. 
\eeq
\fi
Using the notation $\chi  = [\chi_1, \chi_2, \chi_3, \chi_4]$, we rewrite \eqref{EL} in the system form
\beq  \label{odesystem}
\chi' = F(\chi)
\eeq
where
\[
F(\chi) = \begin{bmatrix}
      \chi_2          \\[0.3em]
      \chi_3\\[0.3em]
     \chi_4 \\[0.3em]
	\ -\frac{1}{2} W'(\chi_1) - q \chi_2
     \end{bmatrix}
\]
\iffalse
By the Fundamental Theorem of Calculus
\[
\begin{aligned}
W'(t) - W'(1) - W''(1) \int_1^t 1 ds &= \int_1^t \left( W''(s) - W''(1)  \right) ds \\
&\leq L \int_1^t |s-1| ds = o((t-1)^2),
\end{aligned}
\]
where we have used the fact that $t \mapsto W''(t)$ is Lipschitz with Lipschitz constant $L$, so that $g(s) = o(s^2)$. 
\fi
and the Jacobian of $F$ at $0$ is given by
\[
DF(0) =  \left[ \begin{array}{cccc}
0 & 1 & 0 & 0 \\
0 & 0 & 1 & 0 \\
0 & 0 & 0 & 1 \\
- \frac{1}{2} W''(1) & 0 & -q & 0
\end{array} \right].
\]
The eigenvalues of $DF(0)$ are the roots of the characteristic polynomial
\[
2r^4 + 2qr^2 + W''(1) = 0.
\]
In particular,
\[
r^2 = \frac{-2q \pm \sqrt{ 4q^2 - 8W''(1)  } }{4},
\]
and since $q > 0$ is small, the expression under the square root is negative. We write
\[
\left\{
\begin{aligned}
r^2 = \frac{-2q + \sqrt{ 4q^2 - 8W''(1)  } }{4}, \\
r^2 = \frac{-2q - \sqrt{ 4q^2 - 8W''(1)  } }{4}
\end{aligned}
\right.
\] 
and let $r_1, r_2$ be the roots of the first equation, $r_3, r_4$ those of the second one. We recall that
\[
\sqrt{a + ib} = \pm (\gamma + i \delta),
\]
for
\[
\gamma = \sqrt{ \frac{a + \sqrt{a^2 + b^2} }{2}   }, \quad \delta = \text{sgn}(b) \sqrt{ \frac{-a + \sqrt{a^2 + b^2} }{2}   }.
\]
In the case of $r_1$, we write
\[
r_1 = \left( -\frac{q}{2} + i \frac{\sqrt{2W''(1) - q^2}}{2}   \right)^{1/2},
\]
and a simple calculation shows that
\beq \label{defGamma}
\gamma = \frac{1}{2} \left(  -q + \sqrt{2W''(1)}  \right)^{1/2}, \quad \delta = \frac{1}{2} \left( q + \sqrt{2W''(1)}  \right)^{1/2}.
\eeq
Similarly, one can show that
\beq \label{evalues}
\left\{
\begin{aligned}
r_1 &= \gamma + i \delta, \\
r_2 &= - r_1, \\
r_3 &= \gamma - i \delta, \\
r_4 &= -r_3,
\end{aligned}
\right.
\eeq
%and in turn, we have found that $\{ e^{r_ix} \}_{i=1}^{4}$ are solutions of the linear part of \eqref{EL}.
Applying Lemma \ref{LemmaODEclose} on the interval $\left(c, \frac{y_2 - y_1 - c \eps}{\eps} \right) \subset \left( \frac{\tilde{y}_1 - y_1}{\eps},  \frac{\tilde{y}_2 - y_1}{\eps} \right)$ yields
\beq 
\left|\varphi \left(\frac{y_2 - y_1}{2\eps} \right) \right| \le C(\gamma, \delta) \exp\left(-\gamma  \frac{y_2 - y_1 - 2 c\eps }{2\eps} \right) 
\le C(\gamma, \delta) \exp\left(\gamma c -\gamma  \frac{d}{2\eps} \right) 
\eeq
and  \eqref{A3Res1}, \eqref{A3Res2} follow from definition of $\varphi$ and the fact that
\beq
 \hat{w}\left(\frac{y_1 + y_2}{2} \right) = \hat{v}\left(\frac{y_2 - y_1}{2\eps} \right). 
\eeq
\end{proof}

%
%
%\subsection{Gluing using $C^1$ linearization}
%\textcolor{red}{TO BE DELETED??}
%Consider the jump function $v: [\alpha, \beta] \to \{ \pm 1 \}$ such that $v \equiv -1$ in $[\alpha, a]$ and $v \equiv 1$ in $[a, 0]$, where $\alpha < a < 0$ and we define $d := |a|$. \textcolor{red}{so: $d$ is the distance between two consecutive jumps. or, it is the distance from the left endpoint of the interval to the jump.}  \\
%
%Fix $\e > 0$ and find $y_{1,1} \in (a, a + C_2\e)$ and $y_{1,2} \in (-C_2\e, 0)$, with $C_2 > C_1$, and $y_{1,i} \notin J$, see Lemma \ref{2.2}. In turn,
%\[
%| y_{1,1} - y_{1,2} | \geq d - \e.
%\]
%We now use Lemma \ref{LemmaODEclose} with $T = d -\e$ to find a point $z_1$ (corresponding to $T/2$ in the statement of Lemma \ref{LemmaODEclose}) such that \textcolor{red}{write in terms of Corollary \ref{A3}}
%\[
%|\vp(z_1)| \leq C \exp\left(  -\frac{\gamma (d-\e)}{2} \right),
%\]
%where $\vp = [v, v', v'', v''']$. Rescaling by $\e$, we have
%\[
%|\vp(\tilde z_1)| \leq C \exp\left(  -\frac{\gamma (d-\e)}{2\e} \right) = C \exp\left( \frac{\gamma}{2}  \right) \exp\left(  -\frac{\gamma d}{2\e} \right),
%\]
%where $\tilde z_1$ is the rescaled version of $z_1$. Now we can proceed as in \cite{kal}, steps 3,4 page 44.
%
%
\begin{proof}[Proof of Theorem \ref{thm1}.]
Without loss of generality we can assume that $N(v) \geq 2$
%
%Since $w \in H^2((a,b))$ by Sobolev embedding $w \in C([a,b])$ and for every $k = 1, \ldots N(v)$ there exists a sequence of points $(x_k(w)) \subset (a,b)$ such that
%\beq \label{closePoints}
%| x_k(v) - x_k  | \leq \delta := \min d_k(v)/2
%\eeq
%and
%\beq \label{zeros}
%w(x_k) = 0.
%\eeq
%We set $x_{N+1}^\e := x_1^\e$, {\color{red} in the Bellettini paper the domain is a torus, we have to be more careful here or switch to a torus}
%\[
%\mathcal{S}_\pm(v) := \left\{ k \in \{ 1, \ldots, N(v) \}: x_k(v) \mbox{ is a jump from } \mp 1 \mbox{ to } \pm 1 \right\},
%\]
and define
\beq
\mathcal{M}_k := \{w \in H^2((x_k, x_{k+1})): w(x_k) = 0, \,  \, w(x_{k+1})= 0 \}. %, \, w > 0 \text{ on } (x_k, x_{k+1}) \}.
\eeq
We define $\hat w_k \in H^2((x_k, x_{k+1}))$, for $1 \le k \le N$, to be the minimizer of $E_\e(\cdot, (x_k, x_{k+1} ))$ over $\mathcal{M}_k$.
We also let $\hat w_0 := \hat w_N$.
In turn, $\hat w_k$ solves the Euler--Lagrange equation \eqref{ELEp} with
\[
\begin{aligned}
&\hat w_k(x_k) = \hat w_k(x_{k +1}) = 0.\\
\end{aligned}
\]
Define $d_k := x_{k+1} - x_k$ for $k = 1, \ldots, N(v)$ and
% notice that \eqref{closePoints} implies 
%\[
%|d_k - d_k(v) | \leq 2 \delta.
%\]
%Moreover, we define
\[
I_k^{-}(x_k) := \left( x_k - \frac{d_{k-1}}{2}, x_k \right) \quad \text{ and } \quad I_k^{+}(x_k)  := \left(x_k, x_k + \frac{d_{k}}{2}  \right).
\]
\begin{figure}[h]
%\centering
%{
\raggedleft %replace with \centering{} to center.
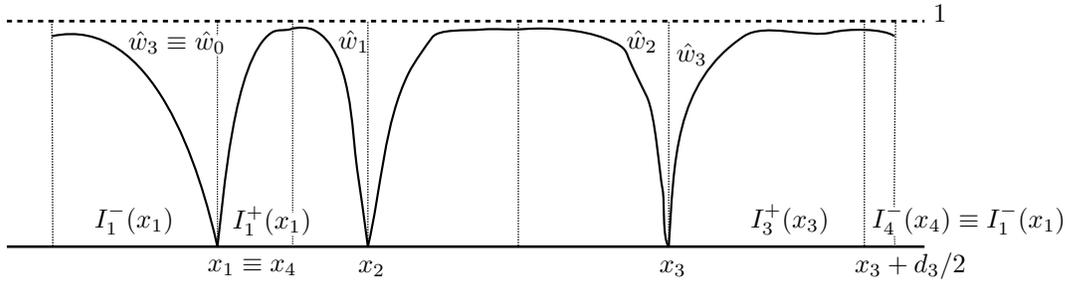 %\resizebox{75mm}{!}{}
\caption{$\hat{w}_k$ and $I^\pm_k$}
\label{fig3}
%}
\end{figure}

From the minimality of $\hat w_k$, we have
\beq \label{energyDown}
\begin{aligned}
E_\e(w; \T) &= \sum_{k=1}^{N(v)} E_\e(w; (x_k, x_{k+1})) \geq \sum_{k=1}^{N(v)} E_\e(\hat w_k; (x_k, x_{k+1})) \\
&= \sum_{k=1}^{N(v)} E_\e(\hat w_{k-1}; I_k^{-}(x_k)) + E_\e(\hat w_k; I_k^{+}(x_k)),
\end{aligned}
\eeq
where in the last equality we have used the fact that $x_{N+1} := x_1$. To complete the proof, it remains to show that
\beq
E_\e(\hat w_{k-1}; I_k^{-}(x_k))  \ge \frac{m_1}{2} - C \exp \left(- \frac{d_{k-1} \gamma}{\e} \right)
\eeq
and
\beq
\label{EnILB}
E_\e(\hat w_{k}; I_k^{+}(x_k))  \ge \frac{m_1}{2} - C \exp \left(- \frac{d_k \gamma}{\e} \right).
\eeq
We will only prove \eqref{EnILB}, the proof of the first inequality being analogous.  Applying the change of variables $z := \frac{x - x_k}{\e}$ gives
\[
\begin{aligned}
E_\e(\hat w_k; I_k^{+}(x_k)) &= \int_{\frac{1}{\e}I_k^{+}(0)} \left( W(\hat w_k (x_k + \e z) ) - \e q |\hat w_k' (x_k + \e z) |^2 + \e^3 |\hat w_k'' (x_k + \e z)|^2   \right) \e dz \\
& = \int_{\frac{1}{\e}I_k^{+}(0)} \left( W(\hat v_k (z) ) - q |\hat v_k' (z) |^2 + |\hat v_k'' (z)|^2   \right) dz = E\left( \hat v_k;  \frac{1}{\e} I_k^{+}(0) \right),
\end{aligned}
\]
where $E(\cdot; \cdot)$ is the rescaled functional defined in \eqref{1DF} and 
\[
\hat v_k(z) := \hat{w}_{k}(x) \mbox{ on each } \frac{1}{\e} I_k^{+}(0).
\]
In addition, we notice that $\hat v_k(0) = \hat w_k(x_k) = 0$ for $1 \le k \le N$ and Proposition \ref{A3}, together with the change of variables we performed, gives
\beq \label{middleCondition}
| \hat v_k(d_k/2\e) - s_k |  = |  \hat{w}_k((x_k + x_{k+1})/2) - s_k|   \leq   C_f \exp\left(-\frac{d_k\gamma}{2\eps} \right)
\eeq
and
\beq \label{middleCondition2}
| \hat v_k' (d_k/2\e)  |  = |  \hat{w}_k'((x_k + x_{k+1})/2) |   \leq   C_f \exp\left(-\frac{d_k\gamma}{2\eps}\right),
\eeq
where $s_k$ is equal to either $1$ or $-1$.
%\item \textcolor{red}{to be checked: } $\hat v_\e$ converges uniformly to $y \mapsto u^\pm(y)$ on compact subsets of $\R^\pm$ as $\e \to 0^+$ for $k \in \mathcal{S}_+(v)$, where $u^\pm$ are the solutions of \eqref{m1} and \eqref{m1} respectively;
%\item $\displaystyle \lim_{\e \to 0^+} E\left[ \hat v_\e,  \frac{1}{\e} I_k^{\e,\pm}(0) \right] = m_{\pm}$ for any $k = 1, \ldots, N(v)$.
We claim that
\beq \label{energyDown3}
E\left( \hat v_k;  \frac{1}{\e} I_k^{\e, +}(0) \right) \geq \frac{m_1}{2}  - E\left(\eta_k;  \R^+ \right)
\eeq
where
\beq
\begin{aligned}
\eta_k(x) := s_k &+ (\hat{v}_k(d_k/2\e) - s_k) \exp\left(-\gamma x\right) \cos (\delta x) \\
&+ \frac{\hat{v}_k'(d_k/2\e) + \gamma (\hat{v}_k(d_k/2\e) - s_k)}{\delta} \exp(-\gamma x) \sin(\delta x).
\end{aligned}
\eeq
%and $\eta_{\e,k}^{+}$ is the minimizer of
%\begin{align} \label{min}
% \inf \Bigg\{ \int_{d_k/2\e}^{\infty} W(u) - q(u')^2 + (u'')^2: \quad & u(d_k/2\e) = \hat v_{\e,k} (d_{k}/2\eps), u'(d_k/2\e) = \hat v_{\e, k}'(d_{k}/2\eps), \nonumber \\
%& u \in H^2_{loc} [d_k/2\e, \infty), \lim_{x \to \infty} u(x) = s_k  \Bigg \}.
%\end{align}
Indeed, let $\theta_\e^{+} \in H_{loc}^2(\R^+)$ be the function that coincides with $\hat v_k$ on $\frac{1}{\e} I_k^{\e, +}(0)$ and $\eta_k^+ := \eta_k(\cdot - d_k/2\e)$ on $\mathbb{R}^+ \backslash \frac{1}{\e} I_k^{\e, +}(0)$.
%\beq 
%2\eta_\eps^{(iv)} + 2q\eta_\eps'' + W'(\eta_\eps ) = 0, \quad \mbox{ in } Q_k^{\e,-},
%\eeq
%\beq
%\quad \eta_\eps(a_k^\e) = \hat v(a_k^\e)\,, \eta_\e'(a_k^\e) = \hat v'(a_k^\e)\,, \eta_\e''(a_k^\e) = p_3 \,, \eta_\e'''(a_k^\e) = p_4,
%\eeq
%where $p_3$ and $p_4$ will be selected appropriately below.
Then,
\[
E(\theta_\e^{+}; \R^+) \geq m_1/2, % m_\pm,
\]
and in turn \eqref{energyDown3} follows.
We now want to find an upper bound for $E(\eta_k; \R^+)$, for $\e$ small enough. \\
The bounds \eqref{middleCondition}, \eqref{middleCondition2} and the definition of $\eta_k$ imply that there exists a constant $C > 0$ such that
\beq
 |\eta_k(x) - s_+|  +   |\eta_k'(x)| +  |\eta_k''(x)| \le C \exp\left( -\frac{d_k \gamma}{2\e} \right) \exp(-\gamma x) \quad \text{ for all } x > 0
\eeq
and consequently
\[
\begin{aligned}
E(\eta_k; \R^+) & = \int_{0}^{\infty} W(\eta_k) - q|\eta_k'|^2 + |\eta_k''|^2 dx \nonumber \\
& = \int_{0}^{\infty} \frac{W''(s_+)}{2} (\eta_k-s_+)^2 - q |\eta_k'|^2 + |\eta_k''|^2 + O((\eta^k-s_+)^3) dx \\
& \le C  \exp\left(-\frac{d_k \gamma}{\e} \right) \int_0^\infty \exp(-2\gamma x) dx \le C  \exp\left(- \frac{d_k \gamma}{\e} \right).
\end{aligned}
\]
\end{proof}

\section{Slow Motion Dynamics: Proof of Theorem \ref{exponentialslowmotion}}

\begin{proof}[Proof of Theorem \ref{exponentialslowmotion}]
Fix $0 < \delta < \min \{1, d/8 \}$. We recall that by definition of $E_{\eps}$
\[
E_{\eps}(u^\e(\cdot,t); \T) = \int_\T \left( \frac{1}{\eps} W(u^\e) - \eps q |u^\e_x|^2 + \eps^3 |u^\e_{xx}|^2  \right) dx.
\]
Integrating by parts and using the regularity of the solution $u^\e$ and equation \eqref{PDE} gives
\[
\begin{aligned}
\frac{d}{dt} E_{\eps}(u^\e(\cdot,t); \T) &=  \int_\T \left( \frac{1}{\eps} W'(u^\e) u^\e_t - 2\eps q u^\e_x u^\e_{xt} + 2\eps^3 u^\e_{xx} u^\e_{xxt}  \right) dx \\
&= \int_\T \left( \frac{1}{\eps} W'(u^\e) u^\e_t + 2\eps q u_{xx}u^\e_{t} + 2\eps^3 u^\e_{xxxx} u^\e_{t}  \right) dx \\
& = - \int_\T |u^\e_t|^2 dx.
\end{aligned}
\]
It follows that for every $T > 0$,
\beq \label{time}
E_{\eps}(u_{0,\e}; \T) - E_{\eps}(u^\e(\cdot,T); \T)= \frac{1}{\eps} \int_0^T \int_\T |u^\e_t|^2 dxdt.
\eeq
Suppose there exists $T_\e$ such that
\beq
\label{utUB}
\int_0^{T_\e} \int_\T |u^\e_t| dxdt \le \delta
\eeq 
Then,
\beq
\int_\T |u_{0,\e} - u^\e(\cdot, T_\e)| dx = \int_\T \left| \int_0^{T_\e} u^\e_t dt \right| dx \le  \int_\T \int_0^{T_\e} |u^\e_t |dt  dx \le \delta
\eeq
and using \eqref{h1} and the triangle inequality
\beq
\label{uTtoV}
||u^\e(\cdot, T_\e) - v||_{L^1(\T)} \le 2\delta .
\eeq
We claim that $u^\e(\cdot, T_\e)$ has at least $N_\e$ zeros, $\{x_k^\e\}_{k=1}^{N_\e}$ that satisfy $\min_k |x_{k+1}^\e - x_k^\e| \ge d - 4\delta$.

Indeed, consider $x_k$, the $k$--th jump point of $v$.  Since the distance between jump points of $v$ is at least $d$ and $\delta \le d/8$, we know that $v$ is constant on $(x_k -2\delta, x_k)$ and on $(x_k, x_k+2\delta)$ and may assume without loss of generality that its value is equal to $1$ on $(x_k -2\delta, x_k)$ and to $-1$ on $(x_k, x_k+2\delta)$.  It follows from \eqref{uTtoV} that $u^\e(\cdot, T_\e)$ must take a positive value somewhere on $(x_k -2\delta, x_k)$  and a negative value on $(x_k, x_k+2\delta)$.  Hence, there exists a zero $x_k^\e \in (x_k-2\delta, x_k+2\delta)$ of  $u^\e(\cdot, T_\e)$.

Applying H\"older inequality, \eqref{h1}, \eqref{time}, and Theorem \ref{thm1} yields
\beq \label{EquationT}
\begin{aligned}
\frac{1}{T_\e} \left( \int_0^{T_\e} \int_\T |u^\e_t| dxdt  \right)^2 &\leq
\int_0^{T_\e} \int_\T |u^\e_t|^2 dxdt \\
&= \eps \left( E_{\eps}(u_{0,\e}; \T)- E_{\eps}(u^\e(\cdot,T_\e); \T) \right) \\
&\leq \eps \left( E_0(v; \T) + \frac{1}{h(\eps)} - m_1 N_\e + C \sum_{k=1}^{N_\e} \text{exp} \left(- \frac{(x_{k+1}^\e - x_k^\e) \gamma}{ \eps} \right)  \right) \\
&\leq \eps \left( E_0(v; \T) + \frac{1}{h(\eps)} - E_0(v; \T) +  C \text{exp} \left(- \frac{(d - 4\delta) \gamma}{\eps}  \right) \right) \\
&= \e \left( \frac{1}{h(\eps)} + C\text{exp} \left(- \frac{(d - 4\delta) \gamma}{\eps}  \right)  \right)
\end{aligned}
\eeq
and as a consequence,
\beq \label{34}
T_\e \geq \frac{1}{C \eps} \left[ \frac{1}{h(\eps)} + \exp\left(- (d - 4\delta)\gamma/ \eps \right)    \right]^{-1} \left(  \int_0^{T_\e} \int_\T |u^\e_t| dxdt \right)^2.
\eeq
Following the ideas of \cite{Grant}, we prove the existence of $T_\e$ as in \eqref{utUB} by dividing the analysis into two cases: first assume that
\[
\int_0^{\infty} \int_\T |u^\e_t| dxdt > \delta.
\]
Since by \eqref{time} with $T$ replaced by any $S > 0$,
\[
\int_0^S \int_\T |u^\e_t|^2 dxdt	\leq \e E_{\eps}(u_{0,\e}; \T) < \infty
\]
we can choose $T_\e$ such that
\beq \label{Tchoice}
\int_0^{T_\e} \int_\T |u^\e_t| dxdt = \delta,
\eeq 
and thanks to \eqref{Tchoice}, equation \eqref{34} gives
\[
T_\e \geq \frac{\delta^2}{ C \eps \left[ \frac{1}{h(\eps)} + \text{exp} \left(- (d - 4\delta)\gamma/ \eps \right)    \right]  } \geq \frac{\delta^2}{ 2 C \eps} \min\{ h(\eps), \text{exp}((d - 4\delta) \gamma/\eps)  \} =: \Lambda_\e.
\]
In turn, \eqref{utUB} is satisfied and \eqref{EquationT} yields
\beq \label{35}
\int_0^{\Lambda_\e} \int_\T |u^\e_t|^2 dxdt \leq C \eps \left[\frac{1}{h(\eps)} + \text{exp} \left(- (d - 4\delta)\gamma/ \eps \right)   \right].
\eeq
On the other hand, if
\[
\int_0^{\infty} \int_\T |u^\e_t| dxdt \leq \delta,
\]
then \eqref{utUB} holds true for all $T > 0$ and again \eqref{35} follows.
To conclude the proof note that for $\e$ sufficiently small 
\[
s_\e := \delta^2 \min\left\{ h(\eps), \text{exp} ((d - 4\delta) \gamma/\eps)   \right\} \le \Lambda_\e
\]
and H\"older's inequality together with \eqref{35} yield
\[
\begin{aligned}
&\sup_{0 \leq t \leq s_\e  } \int_\T  | u^\e(x,t)  - u_{0,\e}(x)  | dx \leq \int_0^{s_\e}  \int_\T  | u^\e_t| dxdt \\
&\leq \left( \min\left\{ h(\eps), \exp\left( \frac{(d - 4 \delta) \gamma}{\eps} \right) \right\}  \int_0^{s_\e }  \int_\T  | u^\e_t|^2 dxdt     \right)^{1/2} \\
&\leq C \left(   \min\left\{ h(\eps), \exp\left( \frac{(d - 4  \delta) \gamma}{\eps} \right) \right\} \eps \delta^2 \left[\frac{1}{h(\eps)} + \exp  \left(- \frac{(d - 4 \delta)\gamma}{\eps} \right) \right]    \right)^{1/2} \\
&\leq C \sqrt{\eps} \delta.
\end{aligned}
\]
Letting $\eps \to 0^+$ gives \eqref{31}.
\end{proof}

\section{Existence of Solutions Via Minimizing Movements}
\iffalse
We start by recalling the Sobolev embedding theorems.
\begin{thm}
Let $\Omega$ be an open bounded subset of $\R^n$ with a $C^1$ boundary $\partial \Omega$. Let $1 \leq p < \infty$ and $m \geq 0$ be an integer. Then the following hold:
\begin{itemize}
\item[(i)] If $mp < N$, then $W^{m,p}(\Omega) \hookrightarrow L^q(\Omega)$, with $\frac{1}{q} = \frac{1}{p} - \frac{m}{N}$;
\item[(ii)] If $mp = N$, then $W^{m,p}(\Omega) \hookrightarrow L^q(\Omega)$ for all $1 \leq q < \infty$;
\item[(iii)] If $mp > N$, let us set $k = \lfloor m - \frac{N}{p}  \rfloor$ and $\alpha = m - \frac{N}{p} - k$, so that $0 \leq \alpha < 1$. Then $W^{m,p}(\Omega) \hookrightarrow C^{k,\alpha}(\Omega)$, where $C^{k,\alpha}(\Omega)$ stands for the space of all functions $v \in C^k(\Omega)$ with $D^{\ell}v \in C^{0,\alpha}(\Omega)$ for any $\ell$ with $|\ell | = k$.
\end{itemize}
\end{thm}
\begin{rem}
As a consequence of the Theorem above, we have that $H^2(\Omega)$ is continuously embedded in $C( \bar{\Omega})$, provided $N < 4$.
\end{rem}
\fi

We now turn to the existence and regularity of solutions for \eqref{PDE} in the more general case of an open, bounded domain $\Omega \subset \R^d$. We notice that the same proof carries over in the case of the one--dimensional torus $\Omega = \T$, that is, when we deal with periodic Dirichlet boundary conditions, which is the framework in which we have analyzed slow motion of solutions of \eqref{PDE}.
%
%\textcolor{red}{TO BE FIXED HERE: we use $n$ for the dimension and the subscript in sequences. I would change the dimension to $m$, but that would be inconsistent with the previous sections. Also, we don't have enough regularity to say that $\dfrac{\partial^3 u}{\partial \nu^3}(x,t) = 0 $. do we just take it off?}

%, which is a special case of the generalized Swift--Hohenberg equation \textcolor{red}{derivative notation to be decided --partial derivatives wrt x in the ''functional notation''-- inconsistent with \ref{SHOrig}} %which is often written as
%\beq \label{SH}
%u_t = r u - \left( c + \frac{\partial^2}{\partial x^2}  \right)^2 u - N(u),
%\eeq
%where $N(u)$ is a non-linearity, often taken as $N(u) = u^3$. 
%More generally, we are dealing with
%\beq \label{SH2}
%\frac{\partial u}{ \partial t} = r u - \left( c + d\frac{\partial^2}{\partial x^2}  \right)^2 u - N(u),
%\eeq
%for $c, d \in \R$. 
\begin{thm} \label{ExistenceTheorem}
Let $\Omega \subset \R^d$, $d \leq 3$, be an open bounded set with $C^2$ boundary, let $u_0 \in H^2(\Omega)$ and the real valued function $z \mapsto W(z)$ be a double--well potential satisfying hypotheses \eqref{w1}--\eqref{w4}.  
%Then for every $T > 0$ there exists a weak solution $u^\e \in L^{\infty}((0,T); H^2(\Omega)) \cap L^2((0,T); H^4(\Omega))$, with $u^\e_t  \in L^2((0,T); L^2(\Omega))$ of
%\beq
%\label{mainPDE}
%\begin{cases}
%u_t = -\frac{1}{\eps} W'(u) - 2 \eps q \Delta u - 2 \eps^3 \Delta^2 u & \text{in} \ \Omega \times (0,T), \\
%\ \ u(x,0) = u_0(x) & \text{in} \ \Omega,  \\
%\dfrac{\partial u}{\partial \nu}(x,t) = \dfrac{\partial^3 u}{\partial \nu^3}(x,t) = 0   &\text{on} \ \partial \Omega \times (0,T),
%\end{cases}
%\eeq
Then for every $T > 0$ there exists a weak solution $u^\e \in L^{\infty}((0,T); H^2(\Omega))$ in the sense of \eqref{ex18}, with $u^\e_t  \in L^2((0,T); L^2(\Omega))$ of
\beq \label{mainPDE}
\begin{cases}
u_t = -\frac{1}{\eps} W'(u) - 2 \eps q \Delta u - 2 \eps^3 \Delta^2 u & \text{in} \ \Omega \times (0,T), \\
\ \ u(x,0) = u_0(x) & \text{in} \ \Omega, 
\end{cases}
\eeq
%with natural boundary conditions and
such that
\[
\int_{\Omega} u(x,t) dx = \int_{\Omega} u_0(x) dx + \int_0^t \int_{\Omega}  \frac{1}{\eps} W'(u(x,s)) dx ds.
\]
Moreover, the following estimates hold
\[
\begin{aligned}
\int_0^T \int_{\Omega} | u_t (x,t) |^2 dxdt &\leq M_\e \sigma^{-1}, \\
\int_{\Omega} | \nabla u (x,t) |^2 dx &\leq 3M_\e \sigma^{-1}, \\
\int_{\Omega} | \nabla^2 u (x,t) |^2 dx &\leq 3M_\e \sigma^{-1},
\end{aligned}
\]
for $\mathcal{L}^1$ a.e. $t \in (0,T)$, where $\sigma \in (0,1)$ and  
\beq \label{starMM}
M_\e := 2\int_{\Omega} \left( \frac{1}{\eps} W(u_0) +\eps  |\nabla u_0|^2 + \eps^3 |\nabla^2u_0|^2   \right) dx.
\eeq
\end{thm}
%
%\begin{rem}
%We notice that the results of Theorem \ref{ExistenceTheorem} hold in the case where $d = 1$, $\Omega = \T$, with $\T$ the one--dimensional torus, as well. The main difference is in the boundary conditions of \eqref{mainPDE} which would have to be replaced with periodic Dirichlet boundary ones. Since the proof follows in a similar manner, we omit the details.
%\end{rem}

\begin{proof}
{\bf Step 1.} For $\ell \in \N$ we set $\tau := T/\ell$ and subdivide the interval $(0,T)$ into $\ell$ subintervals of length $\tau$,
\[
\tau_0 := 0 < \tau_1 < \ldots < \tau_{\ell} := T,
\]
where $\tau_n := n \tau$ for $n = 1, \ldots, \ell$. For every $n = 1, \ldots, \ell$, we let $u_n \in H^2(\Omega)$ be a solution of the minimization problem
\[
\min_{v \in H^2(\Omega)} J_{\e,n}(v;\Omega),
\]
where
\[
\begin{aligned}
J_{\e,n}(v;\Omega) &:= \int_{\Omega} \left( \frac{1}{\eps} W(v) - \eps q |\nabla v|^2 + \eps^3 |\nabla^2v|^2   \right) dx + \frac{1}{2\tau} \int_{\Omega} (v - u_{n-1})^2 dx \\
& = E_{\eps}(v; \Omega) + \frac{1}{2\tau} \int_{\Omega} (v - u_{n-1})^2 dx.
\end{aligned}
\]
In order to prove the existence of $u_n$, we begin by showing that $J_n$ is non--negative and coercive in $H^2(\Omega)$. We fix $q^* > 0$ such that the interpolation inequality Lemma \ref{interpN} holds in $\Omega$, namely
\[
k\eps^2 \int_{\Omega} |\nabla u |^2 dx \leq \int_{\Omega}\left[ W(u) + \eps^4 |\nabla^2 u|^2   \right] dx, \ -\infty < k \leq q^*,
\]
and we let $\sigma \in (0,1)$ be such that $ (q + \sigma)/(1 - \sigma) < q^*$, so that we can write
\beq \label{ex1}
\begin{aligned}
W(u) - q^2 \eps^2 |\nabla u|^2 + \eps^4 |\nabla^2 u|^2 &= 
(1 - \sigma) \left( W(u) - \frac{q + \sigma}{1 - \sigma}   \eps^2  |\nabla u |^2 + \eps^4 |\nabla^2u|^2 \right) \\
&+ \sigma (W(u) + \eps^2  |\nabla u |^2 + \eps^4 |\nabla^2 u|^2  ),
\end{aligned}
\eeq
and in turn $J_{\e,n}$ is non--negative. 
\iffalse
In particular, it can be shown (see Proposition 4.1 in \cite{fl}) that  
\[
\inf_{u \in \mathcal{A}} E_{\eps}(u) > 0.
\]
\fi
Then by \eqref{w4}, and using the fact that $c_W \leq 1$, we obtain
\beq \label{STARex} 
\begin{aligned}
E_{\eps}(u; \Omega) &\geq \sigma c_W \int_{\Omega} \left( (|u| - 1)^{2} + \eps^2 |\nabla u|^2 + \eps^4 |\nabla^2 u|^2   \right) dx.
\end{aligned}
\eeq
The above chain of inequalities implies that
\[
J_{\e,n}(u; \Omega) = E_{\eps}(u; \Omega) + \frac{1}{2\tau} \int_{\Omega} (v - u_{n-1})^2 dx \to \infty  \quad \text{as }  ||u||_{H^2(\Omega)} \to \infty,
\]
and hence $J_{\eps}$ is coercive in $H^2(\Omega)$.

We now let $\displaystyle m_n := \inf_{v \in H^2(\Omega)} J_{\e,n}(v;\Omega)$, and consider a minimizing sequence $\{ v_k \} \subset H^2(\Omega)$ satisfying
\[
m_n \leq J_{\e,n} (v_k;\Omega) \leq m_n + \frac{1}{k},
\]
so that
\[
\lim_{k \to \infty} J_{\e,n}(v_k;\Omega) = m_n.
\]
It follows from \eqref{STARex} that $\{ v_k \}$ is bounded in $H^2(\Omega)$, and hence there exist a subsequence of $\{ v_k \}$ (not relabeled) and some $u_n \in H^2(\Omega)$ such that
\[
\begin{aligned}
v_k \to u_n \quad &\text{in} \ L^2(\Omega), \\
v_k \to u_n \quad &\text{pointwise a.e. in } \Omega, \\
\nabla v_k \to \nabla u_n  \quad &\text{in} \ L^2(\Omega), \\
\nabla^2 v_k \weak \nabla^2 u_n  \quad &\text{in} \ L^2(\Omega).
\end{aligned}
\]
We claim that the above convergences imply that $J_{\e,n}(u_n;\Omega) = m_n$. Indeed, by Fatou's Lemma and lower semicontinuity of $L^2$ norm with respect to weak convergence, we have
\[
m_n = \liminf_{k \to \infty} J_{\e,n}(v_k;\Omega) \geq J_n(u_n) \geq m_n.
\]
It follows that for all $w \in H^2(\Omega)$ and all $t \in \R$,
\[
J_{\e,n}(u_n;\Omega) \leq J_{\e,n}(u_n + t w;\Omega),
\]
and hence the real valued function $\omega(t) := J_{\e,n}(u_n + t w;\Omega)$ has a minimum at $t=0$, so that $\omega'(0) = 0$. Standard arguments show that for every $w \in H^2(\Omega)$,
\beq \label{ex5}
\begin{aligned}
0 &= \int_{\Omega} \left( \frac{1}{\eps} W'(u_n) w - 2\eps q \nabla u_n \cdot \nabla w + 2\eps^3 \nabla^2u_n \cdot \nabla^2 w    \right) \\
&+ \frac{1}{\tau} \int_{\Omega} (u_n - u_{n-1})w,
\end{aligned}
\eeq
where $W'(u_n) w$ is well--defined by the embedding of $H^2(\Omega)$ into $L^\infty (\Omega)$ for $d \leq 3$, and $ \nabla^2u_n \cdot \nabla^2 w = \sum_{i,j} \frac{\partial^2 u_n}{\partial x_i \partial x_j} \frac{\partial^2 w}{\partial x_i \partial x_j}$ is the Fr\"obenius inner product. In particular, this shows that $u_n$ is a weak solution of the equation
\[
-\frac{1}{\eps} W'(u_n) - 2\eps q \Delta u_n - 2 \eps^3 \Delta^2 u_n = \frac{1}{\tau} (u_n - u_{n-1}) \quad \text{in} \ \Omega.
\]
%\textcolor{blue}{CHECK!! 
%Moreover, $u_n$ satisfies 
%\[
%\dfrac{\partial u_n}{\partial n}(x,t) = \dfrac{\partial^3 u_n}{\partial n^3}(x,t) = 0
%\]
%on $\partial \Omega$ 
Since $\Omega$ has finite measure, choosing $w =1$ in \eqref{ex5} gives
\[
0 = \int_{\Omega} \frac{1}{\eps} W'(u_n) dx + \frac{1}{\tau} \int_\Omega (u_n - u_{n-1}) dx.
\]
{\bf Step 2: Apriori bounds.} For $x \in \Omega$ and $t \in (\tau_{n-1}, \tau_n]$, $n= 1, \ldots, \ell$, we define
\beq \label{ex6}
u^{\tau}(x,t) := u_n(x) + (t - \tau_n) \frac{ u_n(x) - u_{n-1}(x) }{\tau}.
\eeq
The goal of this step is to find apriori bounds on $u^{\tau}$. 

Since $J_{\e,n}(u_n;\Omega) = m_n$, it follows that $J_{\e,n}(u_n;\Omega) \leq J_{\e,n}(u_{n-1};\Omega)$, which implies
\[
\begin{aligned}
\frac{1}{2\tau} \int_{\Omega} (u_n - u_{n-1})^2 dx &\leq  \int_{\Omega} \left( \frac{1}{\eps} (W(u_{n-1}) - W(u_n)) - \eps q (|\nabla u_{n-1}|^2 - |\nabla u_n|^2 ) \right) dx \\
& + \int_{\Omega} \eps^3( |\nabla^2u_{n-1}|^2 - |\nabla^2u_n|^2 ) dx.
\end{aligned}
\]
Summing over $n = 1, \ldots, \ell$, we get
\beq \label{beforeEx7}
\begin{aligned}
\frac{1}{2\tau} \sum_{n=1}^{\ell} \int_{\Omega} (u_n - u_{n-1})^2 dx &\leq \int_{\Omega} \left( \frac{1}{\eps} (W(u_{0}) - W(u_{\ell})) - \eps q (|\nabla u_{0}|^2 - |\nabla u_{\ell}|^2 ) \right) dx \\
& + \int_{\Omega} \eps^3( |\nabla^2u_{0}|^2 - |\nabla^2u_{\ell}|^2 ) dx.
\end{aligned}
\eeq
By the interpolation inequality in Lemma \ref{interpN},
\[
\int_{\Omega} \left( \frac{1}{\eps} W(u_{\ell}) - \eps q |\nabla u_{\ell}|^2 + \eps^3 |\nabla^2u_{\ell}|^2   \right) dx \geq \sigma \int_{\Omega} \left( \frac{1}{\eps} W(u_{\ell}) +\eps  |\nabla u_{\ell}|^2 + \eps^3 |\nabla^2u_{\ell}|^2   \right) dx,
\]
where $\sigma \in (0,1)$ was chosen above. Thus, the previous inequalities imply
\beq \label{ex7}
\begin{aligned}
\frac{1}{2\tau} \sum_{n=1}^{\ell} \int_{\Omega} (u_n - u_{n-1})^2 \ dx &+  \sigma \int_{\Omega} \left( \frac{1}{\eps} W(u_{\ell}) +\eps  |\nabla u_{\ell}|^2 + \eps^3 |\nabla^2u_{\ell}|^2   \right) dx \\
&\leq  \int_{\Omega} \left( \frac{1}{\eps} W(u_0) +\eps  |\nabla u_0|^2 + \eps^3 |\nabla^2u_0|^2   \right) dx = \frac{M_\e}{2},
\end{aligned}
\eeq
see \eqref{starMM}. By \eqref{ex6}, for every $x \in \Omega$ and $t \in (\tau_{n-1}, \tau_n]$,
\beq \label{ex8}
\begin{aligned}
u^{\tau}_t(x,t) &= \frac{ u_n(x) - u_{n-1}(x) }{\tau}, \\
\nabla u^{\tau}(x,t) &= \nabla u_n(x) + (t - \tau_n) \frac{ \nabla u_n(x) - \nabla u_{n-1}(x) }{\tau}, \\
\nabla^2 u^{\tau}(x,t) &= \nabla^2 u_n(x) + (t - \tau_n) \frac{ \nabla^2 u_n(x) - \nabla^2 u_{n-1}(x) }{\tau},
\end{aligned}
\eeq
so that by \eqref{ex7} we have
\beq \label{ex9}
\frac{1}{2} \int_{\Omega_T} \left( u^{\tau}_t(x,t)  \right)^2 dxdt +  \sigma \int_{\Omega} \left( \frac{1}{\eps} W(u_{\ell}) +\eps  |\nabla u_{\ell}|^2 + \eps^3 |\nabla^2u_{\ell}|^2   \right) dx \leq \frac{M_\e}{2}
\eeq
which implies
\beq \label{ex10}
\int_{\Omega_T} \left( u^{\tau}_t(x,t)  \right)^2 dxdt \leq M_\e,
\eeq
for every $\tau > 0$. Since $u^{\tau}$ is absolutely continuous, for every $0 \leq t_1 < t_2 \leq T$,
\beq \label{exb}
\begin{aligned}
\int_{\Omega} \left( u^{\tau}(x,t_2) - u^{\tau}(x,t_1)  \right)^2 dx &= \int_{\Omega}  \left( \int_{t_1}^{t_2} u^{\tau}_t(x,t) dt \right)^2 dx \\
&\leq (t_2-t_1) \int_{\Omega_T} \left(  u^{\tau}_t(x,t) \right)^2 dxdt \\
&\leq M_\e(t_2-t_1). 
\end{aligned}
\eeq
Taking $t_1 = 0$ and noticing that $u^{\tau} (x,0) = u_0(x)$, we get
\beq \label{exa}
\int_{\Omega} \left( u^{\tau}(x,t) - u_0  \right)^2 dx \leq M_\e t
\eeq
for every $\tau > 0$ and all $t \in (0,T)$. In turn, by convexity of the function $z \mapsto z^2$,
\beq \label{ex11}
\int_{\Omega} \left( u^{\tau}(x,t) \right)^2 dx \leq 2M_\e t + 2\int_{\Omega} u_0^2(x) dx
\eeq
for every $\tau > 0$ and all $t \in (0,T)$.

Moreover, by \eqref{ex8}, for $x \in \Omega$ and $t \in (\tau_{n-1}, \tau_n ] $,
\[
\begin{aligned}
| \nabla u^{\tau} (x,t) | &\leq 2| \nabla u_n(x) | + |\nabla u_{n-1}(x)|, \\
| \nabla^2 u^{\tau} (x,t) | &\leq 2| \nabla^2 u_n(x) | + |\nabla^2 u_{n-1}(x)|,
\end{aligned}
\]
and by \eqref{ex7} and arbitrariness of $\ell$ we get
\beq \label{ex12}
\int_{\Omega} | \nabla u^{\tau} (x,t) |^2 dx \leq\frac{3M_\e}{\sigma}, \quad \int_{\Omega} | \nabla^2 u^{\tau} (x,t) |^2 dx \leq \frac{3M_\e}{\sigma}.
\eeq

{\bf Step 3: Convergence as $\tau \to 0^+$.} In the previous step we have shown that $\{ u^{\tau} \}$ is bounded in $L^2((0,T); H^2(\Omega))$ and $\{ u_t^{\tau} \}$ is bounded in $L^2((0,T); L^2(\Omega))$. Since these spaces are reflexive, there exist a subsequence of $\{ u^{\tau} \}$ (not relabeled) and $u$ such that $u^{\tau} \weak u$ in $L^2((0,T); H^2(\Omega))$ and in $H^1((0,T); L^2(\Omega))$. Using the fact that the embeddings $H^2(\Omega) \hookrightarrow H^1(\Omega)$ and $H^1(\Omega) \hookrightarrow L^2(\Omega)$ are compact, it follows by the compactness theorem of Aubin and Lions (see e.g. \cite{Aubin}) and a diagonal argument, that, up to a further subsequence, $u^{\tau} \to u$ in $L^2((0,T); L^2(\Omega))$. In turn, for $\mathcal{L}^1$ a.e. $t \in (0,T)$ we have that $u^{\tau}(\cdot, t) \to u(\cdot, t)$ in $L^2(\Omega)$. We are now ready to let $\ell \to \infty$, or equivalently, $\tau \to 0^+$ in \eqref{ex10}, \eqref{exa}, \eqref{ex12}, and deduce the corresponding apriori bounds. \\

\noindent {\bf Step 4: $u$ is a weak solution of the Swift--Hohenberg equation.}\\
We let $x \in \Omega$ and $t \in (\tau_{n-1}, \tau_{n})$, $n= 1, \ldots, \ell$, and define
\beq \label{ex13}
\tilde{u}^{\tau} (x,t) := u_n(x).
\eeq
We claim that $\tilde{u}^{\tau} \weak u$ in $L^2( (0,T); H^2(\Omega)  )$ as $\tau \to 0^+$. 

Given $t \in (0,T]$, we find $n$ such that $t \in (\tau_{n-1}, \tau_n]$ and we notice that
\[
\tilde{u}^{\tau} (x,t) - u^{\tau} (x,t) = u_n(x) - u^{\tau} (x,t) = u^{\tau} (x, \tau_{n}) - u^{\tau} (x,t).
\]
By \eqref{exb},
\beq \label{ex14}
\begin{aligned}
\int_{\Omega} | \tilde{u}^{\tau} (x,t) - u^{\tau} (x,t) |^2 dx &= \int_{\Omega} |u^{\tau} (x,\tau_{n-1}) - u^{\tau} (x,t)|^2 dx \\
&\leq M_\e (t - \tau_{n-1}) \leq M_\e \tau \to 0,
\end{aligned}
\eeq
as $\tau \to 0^+$. This shows that $\tilde{u}^{\tau}(\cdot, t) - u^{\tau}(\cdot, t) \to 0$ in $L^2(\Omega)$ as $\tau \to 0^+$. Moreover, given $\phi \in L^2(\Omega \times (0,T))$, we have
\beq \label{ex15}
\begin{aligned}
\int_{\Omega_T} \tilde{u}^{\tau} (x,t) \phi(x,t) dx dt &= \int_{\Omega_T} (\tilde{u}^{\tau} (x,t) - u^{\tau} (x,t) )\phi(x,t) dxdt \\
&+ \int_{\Omega_T} u^{\tau} (x,t) \phi(x,t) dx dt.
\end{aligned}
\eeq
By H\"older's inequality and \eqref{ex14}, the first integral on the right-hand side of \eqref{ex15} converges to zero.  Using the fact that $u^{\tau} \weak u$ in $L^2((0,T); H^2(\Omega))$ in the second integral, we deduce that $\tilde{u}^{\tau} \weak u$ in $L^2((0,T); L^2(\Omega))$. \\
Moreover, by \eqref{ex12} and the fact that $\tilde{u}^{\tau} (x,t) = u^{\tau} (x,\tau_{n}  )$ for $t \in (\tau_{n-1}, \tau_n]$,
\beq \label{ex16}
\int_{\Omega} | \nabla \tilde{u}^{\tau} (x,t) |^2 dx \leq \frac{3M_\e}{\sigma}, \quad \int_{\Omega} | \nabla^2 \tilde{u}^{\tau} (x,t) |^2 dx \leq \frac{3M_\e}{\sigma},
\eeq
for all $\tau > 0$ and all $t \in (0,T)$. Hence, up to a subsequence, $\tilde{u}^{\tau} \weak u$ in $L^2((0,T); H^2(\Omega))$. Furthermore, by \eqref{ex5}, for every $w \in L^2((0,T); H^2(\Omega))$,
\[
\begin{aligned}
0 &= \int_{\Omega} \left( \frac{1}{\eps} W'(\tilde{u}^{\tau} (x,t)) w - 2\eps q \nabla \tilde{u}^{\tau} (x,t) \cdot \nabla w + 2\eps^3 \nabla^2 \tilde{u}^{\tau} (x,t) \cdot \nabla^2 w    \right) dx \\
&+ \int_{\Omega}u^{\tau}_t (x,t) w \ dx.
\end{aligned}
\]
Integrating in time over $(t_1,t_2)$ gives
\[
\begin{aligned}
0 &= \int_{t_1}^{t_2} \int_{\Omega} \left( \frac{1}{\eps} W'(\tilde{u}^{\tau} (x,t)) w - 2\eps q \nabla \tilde{u}^{\tau} (x,t) \cdot \nabla w + 2\eps^3 \nabla^2 \tilde{u}^{\tau} (x,t) \cdot \nabla^2 w    \right) dx dt \\
&+ \int_{t_1}^{t_2} \int_{\Omega} u^{\tau}_t w \ dxdt.
\end{aligned}
\]
We note that from \eqref{ex7} we have
\iffalse
\textcolor{red}{
set $a_k := u_k - u_{k-1}$ for all $k \in \{ 0, \dots, \ell \}$, then
\[
\begin{aligned}
\int_\Omega (u_n - u_0)^2 \ dx &= \int_\Omega (u_n - u_{n-1} + u_{n-1} - \ldots + u_1 - u_0)^2 \ dx \\
&= \int_\Omega \left( \sum_{i = 1}^n a_i  \right)^2 dx =  \int_\Omega \sum_{i = 1}^n a_i^2  \ dx + 2 \int_\Omega \sum_{i < j} a_ja_i \ dx \\
&\leq \int_\Omega \sum_{i = 1}^n a_i^2  \ dx + 2 \int_\Omega \sum_{i < j} \left( \frac{1}{2} a^2_j + \frac{1}{2} a^2_i \right) \ dx \\
&\leq \ell  \int_\Omega \sum_{i = 1}^\ell a_i^2  \ dx \leq C \ell \tau = C \ell \frac{T}{\ell} = CT
\end{aligned}
\]
which implies
\[
\int_\Omega | u_n | ^2 \ dx \leq C^\star
\]
for some constant $C^\star > 0$.
}
\textcolor{blue}{
just realized, it should work with convexity as well: 
\[
(a+b+c)^2 = 3^2 \left( \frac{a}{3} + \frac{b}{3} + \frac{c}{3}  \right)^2 \leq \frac{3^2}{3}(a^2 + b^2 + c^2 ) = 3 (a^2 + b^2 + c^2 ) 
\]
in our case, we have $n$ elements, which are $a_k := u_k - u_{k-1}$. Thus
\[
(a_n + \ldots + a_1)^2 = n^2 \left( \frac{a_n}{n} + \ldots + \frac{a_1}{n}  \right)^2 \leq n (a_n^2 + \ldots + a_1^2)
\]
}
\fi

\[
\begin{aligned}
\int_\Omega (u_n - u_0)^2 \ dx &= \int_\Omega (u_n - u_{n-1} + u_{n-1} - \ldots + u_1 - u_0)^2 \ dx \\
&\leq \ell \sum_{k=1}^{\ell} \int_{\Omega} (u_k - u_{k-1})^2 \ dx \leq \ell \tau M_\e = T M_\e
\end{aligned}
\]
where we have used the convexity of the function $z \mapsto z^2$ and the fact that $\tau = T/\ell$, and this implies
\beq \label{H2uniform1}
\int_\Omega | u_n | ^2 \ dx \leq C
\eeq
for some constant $C > 0$. Moreover, arguing as in \eqref{beforeEx7}, it follows that
\[
\int_\Omega \left( \frac{1}{\eps} W(u_{n}) - \e q  |\nabla u_{n}|^2 +  \e^3 |\nabla^2 u_{\ell}|^2 \right) dx  \leq \int_{\Omega} \left( \frac{1}{\eps} W(u_{0}) - \eps q |\nabla u_{0}|^2 + \eps^3  |\nabla^2u_{0}|^2 \right) dx \leq \frac{M_\e}{2}
\]
for all $n \in \{0, \ldots, \ell \}$, and in turn, by the interpolation inequality in Lemma \ref{interpN},
\beq \label{H2uniform2}
\int_\Omega  |\nabla u_{n}|^2 dx \leq C \quad \text{ and } \quad \int_\Omega  |\nabla^2 u_{n}|^2 dx \leq C
\eeq
for some constant $C > 0$ and for all $n \in \{0, \ldots, \ell \}$. Using \eqref{H2uniform1}, \eqref{H2uniform2} and the Sobolev embedding theorem, we have
\beq \label{H2uniform3}
||u_n||_{L^\infty (\Omega)} \leq C ||u_n||_{H^2 (\Omega)} \leq C,
\eeq
where $C > 0$ changes from side to side. By the Mean Value Theorem, \eqref{H2uniform3}, and the fact that $W$ is $C^2$, we deduce
\beq
\begin{aligned}
\int_\Omega (W'(\tilde{u}^\tau(x,t)) - W'(u(x,t))) w dx &\le \max_{-C^\star \leq \xi \leq C^\star} |W''(\xi)| \int_\Omega |\tilde{u}(x,t) - u(x,t)| |w| dx \\
&\leq C \int_\Omega |\tilde{u}(x,t) - u(x,t)| |w| dx.
\end{aligned}
\eeq
Letting $\tau \to 0^+$ and using the facts that $\tilde{u}^{\tau} \weak u$ in $L^2((0,T); H^2(\Omega))$, $u^{\tau} \weak u$ in $H^1((0,T); L^2(\Omega))$ we get
\beq \label{ex17}
\begin{aligned}
0 &= \int_{t_1}^{t_2} \int_{\Omega} \left( \frac{1}{\eps} W'(u(x,t)) w - 2\eps q \nabla u (x,t) \cdot \nabla w + 2\eps^3 \nabla^2 u (x,t) \cdot \nabla^2 w    \right) dx dt \\
&+ \int_{t_1}^{t_2} \int_{\Omega}u_t (x,t) w \ dxdt.
\end{aligned}
\eeq

In particular, let $\{ w_k \} \subset H^2(\Omega)$ be dense. 
%\textcolor{red}{The goal now is to take $w$ to be $w_k$ in \eqref{ex17} and pass to the limit. In order to do so, we make use of the Sobolev embedding theorems: when $0\leq N < 4$, $H^2(\Omega)$ is continuously embedded in $C( \bar{\Omega})$, while if $N \geq 4$, then $H^2(\Omega) \hookrightarrow L^q(\Omega)$ for $q = 2^*$. In turn, requiring $W'(s) \leq s^{2^*}$ would allow us to use the Lebesgue Dominated Convergence theorem to pass to the limit. CHECK!!!!!! A growth of the kind $W'(y) \leq c y^{2(\alpha - 1)}$ should be enough, since $2(\alpha - 1) \leq 2 \leq 2^*$, and thus $L^{2^*}(\Omega) \hookrightarrow  L^{2(\alpha - 1)}(\Omega) $. TO BE COMPLETED, SINCE WE'RE ACTUALLY WORKING WITH SPACES LIKE $L^2((0,T);L^2(\Omega))$, SO WE NEED AN AUBIN--LIONS ARGUMENT.
%}
Using the fact that $u(\cdot, t) \in H^2(\Omega)$ and $\frac{\partial u}{\partial t} \in L^2(\Omega)$ for $\mathcal{L}^1$ a.e. $t \in (0,T)$, by the arbitrariness of $t_1$ and $t_2$, we find that
\[
\begin{aligned}
0 &= \int_{\Omega} \left( \frac{1}{\eps} W'(u(x,t)) w_k - 2\eps q \nabla u (x,t) \cdot \nabla w_k + 2\eps^3 \nabla^2 u (x,t) \cdot \nabla^2 w_k    \right) dx \\
&+  \int_{\Omega}u_t  (x,t) w_k \ dx
\end{aligned}
\]
for $\mathcal{L}^1$ a.e. $t \in (0,T)$, where the measure--zero set depends on $k$. Since $\{ w_k \}$ is countable, we can find a set $E \subset (0,T)$ with $\mathcal{L}^1(E) = 0$ such that the previous equality holds for all $t \in (0,T) \setminus E$ and all $k$.

Since $u(\cdot, t) \in H^2(\Omega)$, then $u(\cdot, t) \in L^\infty(\Omega)$ and, again by Mean Value Theorem and the fact that $W$ is $C^2$, it follows that $W'(u(\cdot, t)) \in L^2(\Omega)$. This, together with the density of $\{ w_k \}$ in $H^2(\Omega)$, and the fact that $u_t  \in L^2(\Omega)$ for $t \in (0,T) \setminus E$, implies that
\beq \label{ex18}
\begin{aligned}
0 &= \int_{\Omega} \left( \frac{1}{\eps} W'(u(x,t)) w - 2\eps q \nabla u (x,t) \cdot \nabla w + 2\eps^3 \nabla^2 u (x,t) \cdot \nabla^2 w    \right) dx \\
&+  \int_{\Omega} u_t  (x,t) w \ dx
\end{aligned}
\eeq
for all $t \in (0,T) \setminus E$ and all $w \in H^2(\Omega)$.
Hence $u$ is a weak solution of equation \eqref{mainPDE} and
since $\Omega$ has finite measure, taking $w = 1$ leads to
\beq \label{ex19}
0 = \int_{\Omega} \frac{1}{\eps} W'(u(x,t)) dx +  \int_{\Omega} u_t  (x,t) \ dx,
\eeq
which implies
\[
\int_{\Omega} u(x,t) dx = \int_{\Omega} u_0(x) dx + \int_0^t \int_{\Omega}  \frac{1}{\eps} W'(u(x,s)) dx ds.
\]
\end{proof}

\section{Appendix}

\subsection{Smooth Linearization Near the Hyperbolic Fixed Point}

In the proof of Lemma \ref{LemmaODEclose} we use the fact that in a sufficiently small neighborhood of the fixed point $x_0$ of the system \eqref{odeMain}, $F$ admits a $C^1$ linearization.  This variant of the classical Hartman--Grobman Theorem is based on the concept of $Q$--smoothness of the Jacobian matrix $DF(x_0)$ introduced in \cite{Sell}.
Following \cite{Sell}, we define
\beq
\label{gammaDef}
\gamma(\la; m):= \la - \sum_{i=1}^4 m_ir_i, \text{ for } \la \in \mathbb{C}, \ m_i \in \mathbb{N}_0,
\eeq
where $r_i$ are the eigenvalues in \eqref{evalues}.
\begin{defin}
A matrix $A$ is said to satisfy the \emph{Sternberg condition of order $N$, $N \geq 2$}, if
\beq \label{sternberg}
\gamma(\la; m) \neq 0, \text{ for all } \la \in \Sigma(A), \text{and for all } m \text{ such that } 2 \leq |m| \leq N,
\eeq
where $|m| := \sum m_i$. We will say that $A$ satisfies the \emph{strong Sternberg condition of order $N$}, if $A$ satisfies \eqref{sternberg} and
\beq \label{strongsternberg}
\rp \gamma(\la; m) \neq 0,
\eeq
for all $\la \in \Sigma(A)$ and all $m$ such that $|m| = N$.
\end{defin}
\begin{defin}
Let $\Sigma^{+}(A)$ and $\Sigma^{-}(A)$ be the set of eigenvalues of $A$ having positive and negative real part respectively. $A$ is said to be \emph{strictly hyperbolic} if
\[
\Sigma^+(A) \neq \emptyset, \quad \Sigma^-(A) \neq \emptyset.
\]
The \emph{spectral spread} of $A$ is defined by
\[
\rho^j := \frac{ \max\{ |\rp \la|: \la \in \Sigma^j(A)  \}    }{  \min\{ |\rp \la|: \la \in \Sigma^j(A)  \}   },
\]
for $j = \pm$.
\end{defin}

\begin{defin} \label{qsmooth}
Let $Q \in \mathbb{N}$ and $A$ be hyperbolic. The \emph{$Q$--smoothness} of $A$ is the largest integer $K \geq 0$ such that
\begin{enumerate}[(i)]
\item $Q - K\rho^- \geq 0$, if $\Sigma^+(A) = \emptyset$;
\item $Q - K\rho^+ \geq 0$, if $\Sigma^-(A) = \emptyset$;
\item there exist $M,N \in \mathbb{N}$ with $Q = M + N$ and $M - K\rho^+ \geq 0$, $N - K\rho^- \geq 0$, when $A$ is strictly hyperbolic.
\end{enumerate}
\end{defin}

% \textcolor{red}{the linear part is defined above as well, avoid repetitions. Maybe it is better to use $F$ instead of $f$, since we are dealing with a system of equations.}
%\[
%Df(0)=  \left[ \begin{array}{cccc}
%0 & 1 & 0 & 0 \\
%0 & 0 & 1 & 0 \\
%0 & 0 & 0 & 1 \\
%- \frac{1}{2} W''(1) & 0 & -q & 0
%\end{array} \right], \quad
%G = \left[
%\begin{array}{c}
%0 \\
%0 \\
%0 \\
%-g(v)
%\end{array}
%\right]
%\]
The following theorem is proved in \cite{Sell} (Theorem 1, page 4).
\begin{thm} \label{sell} 
Let $X$ be a finite dimensional Banach space.  Let $Q \geq 2$ be an integer. Assume $G$ is of class $C^{3Q}$ on $U \subset X$ with $0 \in U$, where $D^pG(0) = 0$ for $p=0,1$. Let $A$ be strictly hyperbolic and assume it satisfies the strong Sternberg condition of order Q. Then
\beq \label{ODE}
x' = Ax + G(x)
\eeq
admits a $C^K$--linearization, where $K$ is the $Q$--smoothness of $A$. In other words, there exists a $C^K$--diffeomorphism between solutions of \eqref{ODE} and solutions of its linear part.
\end{thm}
In fact, as remarked in \cite{Sell}, in the case of $A$ strictly hyperbolic it suffices to assume that $G$ is of class $C^{Q + \max(M,N) + K}$. In the remainder, we show that under the assumptions of Lemma \ref{LemmaODEclose}, the matrix $DF(0)$ satisfies the strong Sternberg condition of order $N=2$ and the $2$-smoothness of $DF(0)$ is $K=1$. 

\begin{lemma} \label{sell2}
Consider the ordinary differential equation
\beq
\label{odesystem2}
x' = F(x),
\eeq
where $F$ is a $C^4$ mapping $\R^4 \rightarrow \R^4$ satisfying $F(0) = 0$.  Assume the linearization $DF(0)$ has four eigenvalues $\pm \gamma \pm \delta i$, where $\gamma \ge \lambda > 0$.  Then, the matrix $DF(0)$
 satisfies the strong Sternberg condition of order $N =2$. Moreover, the $Q$--smoothness of $DF(0)$ is $K = 1$, and \eqref{odeMain} admits a $C^1$--linearization around the hyperbolic fixed point $0$.
\end{lemma}

\begin{proof}
We write \eqref{odesystem2} as
\beq \label{odesystemLin}
x' = DF(0)x + G,
\eeq
where $G(x) := F(x) - DF(0)x$ is of class $C^4$, $G(0) = F(0) = 0$, $DG(0) = DF(0) - DF(0) = 0$ and show that  \eqref{sternberg} and \eqref{strongsternberg} hold, for $N=2$. Recalling \eqref{gammaDef}, we have
\beq
\label{gm1}
\gamma(r_1; m) = (1-m_1)r_1 - m_2 r_2 - m_3 r_3 - m_4 r_4,
\eeq
where $|m| = \sum_{i=1}^4 m_i = 2$ and $r_1 := \gamma + \delta i, r_2 := \gamma - \delta i, r_3 := -\gamma + \delta i, r_4 := -\gamma - \delta i$ are the eigenvalues of $DF(0)$.  Assume, for the sake of contradiction, that $\rp \gamma(r_1; m) = 0$ with $|m| = 2$.
Setting the real part of \eqref{gm1} to $0$ and recalling $|m| = 2$, we have
\beq \label{zero}
\left\{
\begin{aligned}
1 - m_1 - m_2 + m_3 + m_4 &= 0, \\
m_1 + m_2 + m_3 + m_4 = 2,
\end{aligned}
\right.
\eeq
Adding the two equations and dividing by two, one has
\beq
m_3 + m_4 = 1/2,
\eeq
a contradiction since $m_3$ and $m_4$ are integers.
A similar argument for any $\la \in \Sigma(Df(0))$ shows that  \eqref{strongsternberg} and \eqref{sternberg} hold for the matrix $Df(0)$, and $N=2$.  

It remains to show that the $2$-smoothness of $DF(0)$ is $K=1$.
Since $|\rp \la| = \gamma$, for all $\la \in \Sigma(Df(0))$, then the spectral radius of $Df(0)$ is $\rho^i = 1$, for $i= \pm$.
Being $Df(0)$ is strictly hyperbolic, we are in case $(iii)$ of Definition \ref{qsmooth} and $Q = 2$ implies $M = N = 1$. In turn, the largest integer $K$ that satisfies
\beq
\left\{
\begin{aligned}
M - K\rho^+ = 1 - K &\geq 0, \\
N - K\rho^- = 1- K &\geq 0,
\end{aligned}
\right.
\eeq
is $K = 1$, which is then the $2$--smoothness of $Df(0)$.
We now apply Theorem \ref{sell} with $Q = 2$ and $A = DF(0)$ to conclude that \eqref{odesystemLin} admits a $C^1$--linearization.
\end{proof}

\section{Acknowledgments} 
The authors would like to thank Irene Fonseca and Giovanni Leoni for the fruitful discussions and preliminary readings of the manuscript. They would also like to thank the Center for Nonlinear Analysis at Carnegie Mellon University, where part of the research was developed. The results contained in this manuscript fulfill part of the Ph.D. requirements of the second author, whose research was partially supported by the awards DMS 0905778 and DMS 1412095.

\newpage
\bibliographystyle{abbrv}
%\bibliographystyle{alpha}
%\bibliographystyle{acm}
%\bibliography{SlowSH}
\bibliography{FullBibFixed}

\end{document}